\documentclass[12pt]{amsart}
\allowdisplaybreaks
\usepackage[english]{babel}
\usepackage[pdftex,textwidth=400pt,marginratio=1:1]{geometry}
\usepackage{amsfonts}
\usepackage[dvips]{graphics}
\usepackage[colorinlistoftodos]{todonotes}
\usepackage{amsmath}
\usepackage{amsthm}
\usepackage{amssymb}
\usepackage{bbm}
\usepackage{cancel}
\usepackage{color}
\usepackage[curve]{xypic}
\usepackage{graphicx}
\newtheorem{theorem}{Theorem}[section]
\newtheorem{corollary}[theorem]{Corollary}
\newtheorem{proposition}[theorem]{Proposition}
\newtheorem{lemma}[theorem]{Lemma}
\newtheorem{conjecture}[theorem]{Conjecture}

\theoremstyle{definition}

\newtheorem{remark}[theorem]{Remark}

\theoremstyle{property}

\DeclareFontFamily{OT1}{rsfs}{}
\DeclareFontShape{OT1}{rsfs}{n}{it}{<-> rsfs10}{}
\DeclareMathAlphabet{\curly}{OT1}{rsfs}{n}{it}

\newcommand\I{\mathcal I}

\renewcommand\L{\mathcal L}

\renewcommand\O{\mathcal O}
\newcommand\PP{\mathbb P}

\newcommand\EE{\mathbb E}
\newcommand\cA{\mathcal A}

\newcommand\E{\mathbb E}
\newcommand\C{\mathbb C}

\newcommand\sfZ{\mathsf Z}

\newcommand\Q{\mathbb Q}

\newcommand\Z{\mathbb Z}
\newcommand\cZ{\mathcal Z}

\newcommand\Coeff{\mathrm{Coeff}}

\newcommand\inst{\mathrm{inst}}

\newcommand\mov{\mathrm{mov}}

\newcommand\s{\mathfrak s}
\renewcommand\t{\mathfrak t}

\newcommand\SU{\mathrm{SU}}

\newcommand\vd{\mathrm{vd}}
\newcommand\pt{\mathrm{pt}}
\newcommand\vir{\mathrm{vir}}

\newcommand\SW{\mathrm{SW}}

\newcommand\td{\mathrm{td}}
\newcommand\rk{\operatorname{rk}}

\newcommand\tr{\operatorname{tr}}

\newcommand\ch{\operatorname{ch}}

\newcommand\id{\operatorname{id}}

\newcommand\mon{\operatorname{mon}}

\newcommand\Hom{\operatorname{Hom}}
\renewcommand\hom{\mathcal{H}{\it{om}}}

\newcommand\Pic{\operatorname{Pic}}

\newcommand\Sym{\operatorname{Sym}}

\newcommand\mdot{{\scriptscriptstyle\bullet}}

\newcommand\INTO{\ar@{^{(}->}[r]}

\DeclareRobustCommand{\SkipTocEntry}[4]{}

\begin{document}
\title[Verlinde formulae on complex surfaces]{Verlinde formulae on complex surfaces: $K$-theoretic invariants}
\author[G\"ottsche, Kool, Williams]{L.~G\"ottsche, M.~Kool, and R.~A.~Williams}
\maketitle

\vspace{-0.8cm}

\begin{abstract}
We conjecture a Verlinde type formula for the moduli space of Higgs sheaves on a surface with a holomorphic 2-form. The conjecture specializes to a Verlinde formula for the moduli space of sheaves. Our formula interpolates between $K$-theoretic Donaldson invariants studied by the first named author and Nakajima-Yoshioka and $K$-theoretic Vafa-Witten invariants introduced by Thomas and also studied by the first and second named authors. We verify our conjectures in many examples (e.g.~on K3 surfaces).
\end{abstract}
\thispagestyle{empty}
%\tableofcontents

\section{Introduction} \label{intro}

Let $C$ be a smooth projective curve of genus $g \geq 2$ over $\C$. The classical $\theta$-functions at level $k \geq 1$ for $C$ are defined as follows. The map $C \rightarrow \Pic^1(C)$, $p \mapsto [\O_C(p)]$ gives rise to the Abel-Jacobi map on the symmetric product
$$
\Sym^{g-1}(C) \rightarrow \Pic^{g-1}(C)
$$
and the image $\Theta$, which has codimension one, is known as the theta divisor. Denote by $\L$ the corresponding line bundle. The $\theta$-functions of level $k$ are defined as the elements of $H^0(\Pic^{g-1}(C), \L^{\otimes k})$. Since $H^{>0}(\Pic^{g-1}(C), \L^{\otimes k})=0$, the Riemann-Roch theorem gives the dimension of the space of $\theta$-functions of level $k$ as the degree of $\exp(k \Theta)$. Since $\Theta^g / g! = 1$, one obtains 
$$
\dim H^0(\Pic^{g-1}(C), \L^{\otimes k}) = k^g. 
$$
The Verlinde formula extends this equation to moduli spaces of rank 2 (and higher) stable vector bundles on $C$ as follows. See \cite{Bot} for a survey.

Denote by $M:=M_C(2,0)$ the moduli space of rank 2 semistable vector bundles $E$ on $C$ with $\det E \cong \O_C$. Then $\Pic(M)$ is generated by the determinant line bundle $\L$ \cite{Bea,DN}. The Verlinde formula (for rank 2 and trivial determinant), originating from conformal field theory \cite{Ver}, is the following expression
\begin{equation*} 
\dim H^0(M,\L^{\otimes k}) = \Bigg( \frac{k+2}{2} \Bigg)^{g-1} \sum_{j=1}^{k+1} \sin\Big( \frac{\pi j}{k+2} \Big)^{2-2g}.
\end{equation*}
This formula has been proved by several people \cite{Sze,BS,Tha,Kir,Don,Ram,DW,Zag2} (for rank 2) and \cite{Fal, BL} (for general rank). Numerical aspects of this formula were studied by D.~Zagier \cite{Zag1}.

Let $N:=N_C(2,0)$ be the moduli space of rank 2 semistable Higgs bundles $(E,\phi)$ on $C$ with $\det E \cong \O_C$. Here $E$ is a rank 2 vector bundle and $\phi : E \rightarrow E \otimes K_C$ is called the Higgs field. The moduli space $N$ is non-compact. It has a $\C^*$-action defined by scaling the Higgs field. The determinant line bundle $\mathcal{L}$ on $N$ is $\C^*$-equivariant, therefore $H^0(N,\L^{\otimes k})$ is a $\C^*$-representation.
Recently, Halpern-Leistner \cite{H-L} and Andersen-Gukov-Du Pei \cite{AGDP} found a formula for $\dim H^0(N, \L^{\otimes k})$, which can be seen as a Verlinde formula for Higgs bundles. In physics, this formula is related to complex Chern-Simons theory of the (3-dimensional) Seifert manifold $C \times S^1$ embedded into string theory \cite{GDP}.

In this paper, we study Verlinde type formulae on the moduli space of rank 2 Gieseker stable (Higgs) sheaves on $S$, where $S$ is a smooth projective surface satisfying $p_g(S)>0$ and $b_1(S)=0$.

\subsection{Verlinde formula for moduli of sheaves}

Denote by $M:=M_S^H(2,c_1,c_2)$ the moduli space of rank $2$ Gieseker $H$-stable torsion free sheaves on $S$ with Chern classes $c_1 \in H^2(S,\Z)$ and $c_2 \in H^4(S,\Z)$. We assume there are no rank 2 strictly Gieseker $H$-semistable sheaves on $S$ with Chern classes $c_1,c_2$. 
Then $M$ is a projective scheme with perfect obstruction theory of virtual dimension
\begin{equation} \label{vd}
\vd = 4c_2-c_1^2-3\chi(\O_S).
\end{equation}
When a universal sheaf $\EE$ exists on $M \times S$, the virtual tangent bundle is given by $T_M^\vir = R \hom_{\pi_M}(\EE,\EE)_0[1]$, where $\pi _M: M \times S \rightarrow M$ denotes projection and $(\cdot)_0$ denotes trace-free part. In general $\EE$ exists only \'etale locally. Nevertheless, $R \hom_{\pi_M}(\EE,\EE)_0[1]$ exists globally on $M \times S$, essentially because this expression is invariant under replacing $\EE$ by $\EE \otimes \mathcal{L}$ for any $\mathcal{L} \in \Pic(M \times S)$ \cite[Sect.~10.2]{HL}. Algebraic Donaldson invariants are defined by integrating polynomial expressions in slant products over $[M]^{\vir}$. These were studied in detail, for any rank, in T.~Mochizuki's remarkable monograph \cite{Moc}.

Let $\alpha \in H^*(S,\Q)$. When a universal sheaf $\EE$ exists on $M \times S$, we consider the $\mu$-insertion defined by the slant product 
\begin{align} 
\begin{split} \label{muinsert}
/ : H^p(S \times M,\Q) \times H_q(S,\Q) \rightarrow H^{p-q}(M,\Q), \\
\mu(\alpha) := \big(c_2(\EE) - \frac{1}{4} c_1(\EE)^2 \big) / \mathrm{PD}(\alpha)  \in H^*(M,\Q),
%\mu(\alpha) := \pi_{M*} \Big( \pi_S^* \alpha \cdot \big(c_2(\EE) - \frac{1}{4} c_1(\EE)^2 \big) \cap [M \times S] \Big) \in H_*(M,\Q),
\end{split}
\end{align}
where $\mathrm{PD}(\cdot)$ denotes Poincar\'e dual. Note that
$$
c_2(\EE) - \frac{1}{4} c_1(\EE)^2 = -\frac{1}{4} \ch_2(\EE \otimes \EE \otimes \det(\EE)^*),
$$
where the sheaf $\EE \otimes \EE \otimes \det(\EE)^*$ always exists globally on $M \times S$, again, essentially because this expression is invariant under replacing $\EE$ by $\EE \otimes \mathcal{L}$. Therefore \eqref{muinsert} is always defined. When $L \in \Pic(S)$ satisfies $c_1(L)c_1 \in 2\Z$, there exists a line bundle $\mu(L) \in \Pic(M)$, whose class in cohomology is \eqref{muinsert} for $\alpha=c_1(L)$ \cite[Ch.~8]{HL}. One refers to $\mu(L)$ as a Donaldson line bundle. The first conjecture concerns 
$$
\chi^{\vir}(M,\mu(L)) := \chi(M, \O^{\vir}_M \otimes \mu(L)),
$$
known as a $K$-theoretic Donaldson invariant \cite{GNY2}.\footnote{When the Donaldson line bundle does not exist, we \emph{define} $\chi^{\vir}(M,\mu(L))$ by the virtual Hirzebruch-Riemann-Roch formula \cite[Cor.~3.4]{FG}, i.e.~$\int_{[M]^{\vir}} e^{\mu(c_1(L))} \, \td(T_M^{\vir})$.} 
Here $\O_M^{\vir}$ denotes the virtual structure sheaf of $M$ \cite[Sect.~3.2]{FG}.
The first named author, H.~Nakajima, and K.~Yoshioka determined their wall-crossing behaviour, when $S$ is a toric surface using the $K$-theoretic Nekrasov partition function \cite{GNY2}. 
For rational surfaces the first named author and Y.~Yuan established structure formulae for these invariants and relations to strange duality \cite{GY, Got1}. 

We denote intersection numbers such as $\int_S c_1(L) c_1(\O(K_S))$ by $c_1(L)c_1(\O(K_S))$ or simply $LK_S$. Denote by $\SW(a)$ the Seiberg-Witten invariant of $a \in H^2(S,\Z)$.\footnote{We use Mochizuki's convention: $\SW(a) = \widetilde{\SW}(2a-K_S)$ with $\widetilde{\SW}(b)$ the usual Seiberg-Witten invariant in class $b \in H^2(S,\Z)$. Moreover, there are finitely many $a \in H^2(S,\Z)$ such that $\SW(a) \neq 0$. Such classes are called Seiberg-Witten basic classes.} 

\begin{conjecture} \label{conj1}
Let $S$ be a smooth projective surface with $p_g(S)>0$, $b_1(S)=0$, and $L \in \Pic(S)$. Let $H,c_1,c_2$ be chosen such that there are no rank 2 strictly Gieseker $H$-semistable sheaves on $S$ with Chern classes $c_1, c_2$. Then $\chi^{\vir}(M_S^H(2,c_1,c_2),\mu(L))$ equals the coefficient of $x^{\vd}$ of
$$
\frac{2^{2-\chi(\O_S)+K_S^2}}{(1-x^2)^{\frac{(L-K_S)^2}{2}+\chi(\O_S)}} \sum_{a\in H^2(S,\Z)} \SW(a) \, (-1)^{ac_1} \, \left(\frac{1+x}{1-x}\right)^{\left(\frac{K_S}{2}-a\right)(L-K_S)}.
$$
\end{conjecture}

In Section \ref{inst} we verify this conjecture in many cases for: $S$ a K3 surface, elliptic surface, Kanev surface, double cover of $\PP^2$ branched along a smooth octic curve, quintic surface, and blow-ups thereof. Our strategy is similar to \cite{GNY3, GK1, GK2, GK3}. We first express $\chi^{\vir}(M, \mu(L))$ in terms of algebraic Donaldson invariants. Using Mochizuki's formula \cite[Thm.~1.4.6]{Moc}, the latter can be written in terms of integrals on Hilbert schemes of points. We show that these integrals can be combined into a generating series which is a cobordism invariant and hence determined on $\PP^2$ and $\PP^1 \times \PP^1$. On $\PP^2$ and $\PP^1 \times \PP^1$, we determine this generating series (to some order) by localization. 

Finally, in Section \ref{appl} we discuss interesting special cases of Conjecture \ref{conj1}.

\subsection{Verlinde formula for moduli of Higgs sheaves}

Let $H$ be a polarization on $S$. Recently Y.~Tanaka and R.~P.~Thomas \cite{TT1} proposed a mathematical definition of $\SU(r)$ Vafa-Witten invariants of $S$. We consider the case $r=2$. Their definition involves the moduli space of Higgs sheaves $(E,\phi)$
$$
N := N_S^H(2,c_1,c_2) = \big\{ (E,\phi) \, : \, \tr \phi = 0, \, c_1(E) = c_1, \, c_2(E) =c_2 \big\},
$$
where $E$ is a rank $2$ torsion free sheaf, $\phi : E \rightarrow E \otimes K_S$ is a morphism, and the pair $(E,\phi)$ satisfies a (Gieseker) stability condition with respect to $H$. Tanaka-Thomas show that $N$ admits a symmetric perfect obstruction theory in the sense of \cite{Beh}. As in the curve case, one can scale a Higgs sheaf by sending $(E,\phi)$ to $(E,t\phi)$ for any $t \in \C^*$. This defines an action of $\C^*$ on $N$. As in the previous section, we assume stability and semistability coincide. 
Then the fixed locus $N^{\C^*}$ is projective and the Vafa-Witten invariants are defined as
\begin{equation*} 
\int_{[N^{\C^*}]^{\vir}} \frac{1}{e(N^\vir)} \in \Q,
\end{equation*}
where $N^{\vir}$ denotes the virtual normal bundle and $e(\cdot)$ is the equivariant Euler class \cite{GP}. The fixed locus $N^{\C^*}$ has two types of connected components:
\begin{itemize}
\item Components containing $(E,\phi)$ with $\phi = 0$, which we refer to as the \emph{instanton branch}. This branch is isomorphic to the Gieseker moduli space $M:=M_S^H(2,c_1,c_2)$. The $\C^*$-localized perfect obstruction theory on $M$ coincides with the one from the previous section.
\item Components containing $(E,\phi)$, where $E = E_0 \oplus E_1 \otimes \mathfrak{t}^{-1}$ is the decomposition of $E$ into rank 1 eigensheaves, and
$\phi : E_0 \rightarrow E_1 \otimes K_S$. Here $\mathfrak{t}$ denotes the weight one character of $\C^*$. These components constitute the \emph{monopole branch}, which we collectively denote by $M^{\mon}$. Denote by $S^{[n]}$ the Hilbert scheme of $n$ points on $S$ and by $|\beta|$ the linear system of an algebraic class $\beta \in H^2(S,\Z)$. A.~Gholampour and Thomas \cite{GT1, GT2} prove that the monopole components are isomorphic to incidence loci\footnote{For fixed $r=2$, $c_1, c_2$, the virtual dimension of $M^{\mon} \subset N^{\C^*}$ is in general \emph{not} given by \eqref{vd}. In fact, $M^{\mon}$ can have components of different virtual dimension (see Remark \ref{differentvd}).}
$$
S_{\beta}^{[n_0,n_1]} :=\{ (Z_0,Z_1,C) \, : \, I_{Z_0}(-C) \subset I_{Z_1} \} \subset S^{[n_0]} \times S^{[n_1]} \times |\beta|,
$$
for certain $n_0,n_1,\beta$, where $I_Z \subset \O_S$ is the ideal sheaf corresponding to $Z \subset S$. Moreover, they show that the $\C^*$-localized perfect obstruction theory on $S_{\beta}^{[n_0,n_1]}$ is naturally obtained by realizing this space as a degeneracy locus inside the smooth space $S^{[n_0]} \times S^{[n_1]} \times |\beta|$ and reducing the perfect obstruction theory coming from this description (Section \ref{mon}). 
\end{itemize}

Let $M' \subset M^{\mon}$ be a connected component of the monopole branch and let $\alpha \in H^*(S,\Q)$. Similar to the previous section, we define
\begin{equation} \label{muinsertmon}
\mu(\alpha) := \big(c_2^{\C^*}(\EE) - \frac{1}{4} c_1^{\C^*}(\EE)^2 \big) / \mathrm{PD}(\alpha) \in H_{\C^*}^{*}(M',\Q), 
%\mu(\alpha) := \pi_{M' *} \Big( \pi_S^* \alpha \cdot \big(c_2^{\C^*}(\EE) - \frac{1}{4} c_1^{\C^*}(\EE)^2 \big) \cap [M' \times S] \Big) \in H^{\C^*}_{*}(M',\Q), 
\end{equation}
where the Chern classes are $\C^*$-equivariant, $M'$ and $S$ carry the trivial torus action, and $\EE$ is the universal sheaf on $M' \times S$.

Vafa-Witten invariants can also be seen as reduced Donaldson-Thomas invariants counting 2-dimensional sheaves on $X = \mathrm{Tot}(K_S)$ ---the total space of the canonical bundle on $S$ \cite{GSY2}. From this perspective, it is more natural to work with the Nekrasov-Okounkov twist of $\O_N^{\vir}$ \cite{NO}, which is defined as
$$
\widehat{\O}_N^{\vir} := \sqrt{K_N^{\vir}} \otimes \O_N^{\vir},
$$
where $\sqrt{K_N^{\vir}}$ is a choice of square root of $K_N^{\vir} = \det(\Omega_N^{\vir})$. Over the fixed locus $N^{\C^*}$, this choice of square root is canonical \cite[Prop.~2.6]{Tho}.
For any (possibly infinite-dimensional) graded vector spaces set
$$
\chi\Big(\bigoplus_i \t^{a_i} - \bigoplus_j \t^{b_j} \Big) := \sum_i y^{a_i} - \sum_j y^{b_j}.
$$ 
The $K$-theoretic Vafa-Witten invariants are \cite[(2.12), Prop.~2.13]{Tho}
\begin{equation*} 
\chi(N,\widehat{\O}^{\vir}_N) := \chi( R\Gamma(N, \widehat{\O}_N^{\vir}) ) = \chi \Big(N^{\C^*}, \frac{\O^{\vir}_{N^{\C^*}}}{\Lambda_{-1} (N^{\vir})^{\vee}} \otimes \sqrt{K_N^{\vir}} \Big|_{N^{\C^*}} \Big).
\end{equation*} 
Here $\Lambda_{-1} (\cdot)$ is introduced in Section \ref{inst} and $y$ is related to $t := c_1^{\C^*}(\mathfrak{t})$ by $y=e^t$. The Nekrasov-Okounkov twist ensures that these invariants are unchanged under $y \leftrightarrow y^{-1}$ \cite[Prop.~2.27]{Tho}. Our next two conjectures concern
\begin{equation*}
\chi(N, \widehat{\O}^{\vir}_N \otimes \mu(L)) := \chi( R\Gamma(N, \widehat{\O}_N^{\vir} \otimes \mu(L))),
\end{equation*}
where $L \in \Pic(S)$.\footnote{If the line bundle $\mu(L)$ does not exist on $N$ (or $N^{\C^*}$), then we \emph{define} these invariants by virtual $\C^*$-localization combined with the virtual HRR formula as before.} This expression has instanton and monopole contributions corresponding to the decomposition $N^{\C^*} = M \sqcup M^{\mon}$.
By the argument in \cite[Sect.~2.5]{Tho}, the instanton contribution equals
$$
(-1)^{\vd} y^{-\frac{\vd}{2}} \chi_{-y}^{\vir}(M,\mu(L)) := (-1)^{\vd} y^{-\frac{\vd}{2}} \sum_{p}(-y)^p \chi^\vir(M,\Lambda^p \Omega^{\vir}_M \otimes \mu(L)),
$$
where $\vd$ is given by \eqref{vd} and $\chi_y^\vir(M,\cdot)$ is the twisted virtual $\chi_y$-genus \cite{FG}. 

Consider the following two theta functions and the normalized Dedekind eta function
\begin{equation} \label{thetadef}
\theta_3(x,y) = \sum_{n \in \Z} x^{n^2} y^n, \quad \theta_2(x) = \sum_{n \in \Z + \frac{1}{2}} x^{n^2} y^n, \quad \overline{\eta}(x) = \prod_{n=1}^{\infty} (1-x^n).
\end{equation}
We also use the following notation. For any $a,b \in H^2(S,\Z)$, define 
\begin{equation} \label{defdelta}
\delta_{a,b} =  \# \left\{ \gamma \in H^2(S,\Z) \,: \, a-b =2\gamma \right\}.
\end{equation}
\begin{conjecture} \label{conj2}
Let $S$ be a smooth projective surface with $p_g(S)>0$, $b_1(S)=0$, and $L \in \Pic(S)$. Let $H,c_1,c_2$ be chosen such that there are no rank 2 strictly Gieseker $H$-semistable sheaves on $S$ with Chern classes $c_1, c_2$. Let $\vd$ be defined by \eqref{vd}. Then $y^{-\frac{\vd}{2}} \chi_{-y}^{\vir}(M_S^H(2,c_1,c_2),\mu(L))$ equals the coefficient of $x^{\vd}$ of
\begin{align*}
&4\left( \frac{1}{2} \prod_{n=1}^{\infty} \frac{1}{(1-x^{2n})^{10}(1-x^{2n}y)(1-x^{2n}y^{-1})}\right)^{\chi(\O_S)}\left(\frac{ 2 \overline \eta(x^4)^2}{\theta_3(x,y^{\frac{1}{2}})}\right)^{K_S^2}\\
&\cdot\left(\prod_{n=1}^{\infty} \left(\frac{(1-x^{2n})^2}{(1-x^{2n}y)(1-x^{2n}y^{-1})}\right)^{n^2}\right)^{\frac{L^2}{2}}
\left(\prod_{n=1}^{\infty}\left( \frac{1-x^{2n}y^{-1}}{1-x^{2n}y}\right)^n\right)^{LK_S}\\
&\cdot\sum_{a\in H^2(S,\Z)} (-1)^{c_1 a} \, \SW(a) \,   \left(\frac{\theta_3(x,y^{\frac{1}{2}})}{\theta_3(-x,y^{\frac{1}{2}})}\right)^{aK_S} \\
&\cdot \Bigg(\prod_{n=1}^{\infty}\left(\frac{(1-x^{2n-1} y^{\frac{1}{2}})(1+x^{2n-1}y^{-\frac{1}{2}})}{(1-x^{2n-1}y^{-\frac{1}{2}})(1+x^{2n-1} y^{\frac{1}{2}})}\right)^{2n-1}\Bigg)^{\frac{L(K_S-2a)}{2}}.
\end{align*}
\end{conjecture}
 \begin{conjecture} \label{conj3}
Let $S$ be a smooth projective surface with $p_g(S)>0$, $b_1(S)=0$, and $L \in \Pic(S)$. Let $H,c_1,c_2$ be chosen such that there are no rank 2 strictly Gieseker $H$-semistable Higgs sheaves on $S$ with Chern classes $c_1,c_2$. Let  $N:=N_S^H(2,c_1,c_2)$ and let $\vd$ be defined by \eqref{vd}. Then the monopole contribution to $\chi(N, \widehat{\O}^{\vir}_N \otimes \mu(L))$ equals the coefficient of $(-x)^{\vd}$ of
\begin{align*}
&\left( \prod_{n=1}^{\infty} \frac{1}{(1-x^{8n})^{10}(1-x^{8n}y^2)(1-x^{8n}y^{-2})}\right)^{\chi(\O_S)} \left( \frac{\overline \eta(x^4)^2}{\theta_2(x^4,y)}\right)^{K_S^2}\\
&\cdot\left(\prod_{n=1}^{\infty} \left(\frac{(1-x^{8n})^2}{(1-x^{8n}y^2)(1-x^{8n}y^{-2})}\right)^{n^2}\right)^{2L^2}
\left( \prod_{n=1}^{\infty}\left( \frac{1- x^{4n}y^{-1}}{1- x^{4n}y}\right)^n\right)^{2LK_S} \\
&\cdot\sum_{a\in H^2(S,\Z)} \delta_{c_1,K_S-a} \, \SW(a) \, k_a  \left(\frac{\theta_2(x^4,y)}{\theta_3(x^4,y)}\right)^{aK_S} \left( \prod_{n=1}^{\infty}\left( \frac{1+x^{8n-4}y^{-1}}{1+ x^{8n-4}y}\right)^{2n-1}\right)^{2aL} \\
&\cdot \left(\prod_{n=1}^{\infty}\left( \frac{1+x^{8n}y^{-1}}{1+x^{8n}y}\right)^n\right)^{4L(K_S-a)} \cdot \left(\prod_{n=1}^{\infty} \left(\frac{1+x^{4n}y^{-1}}{1+x^{4n} y}\right)^{n}\right)^{LK_S},
\end{align*}
where $k_a := x^{- 3 \chi(\O_S)} (y^{\frac{1}{2}} + y^{-\frac{1}{2}})^{- \chi(\O_S)} y^{\frac{1}{2} L(a-K_S)}$.
\end{conjecture}

Together these two conjectures give a Verlinde type formula for the moduli space of Higgs sheaves on a surface $S$ satisfying $b_1(S) = 0$ and $p_g(S)>0$. Moreover our formulae interpolate between the following two invariants:
\begin{itemize}
\item \textbf{$K$-theoretic Donaldson invariants.} After replacing $x$ by $x y^{\frac{1}{2}}$ in the formula of Conjecture \ref{conj2}, we can set $y=0$. This replacement  provides a formula for $\chi_{-y}^{\vir}(M,\mu(L))$ and setting $y=0$ implies the formula for $K$-theoretic Donaldson invariants of Conjecture \ref{conj1}.
\item \textbf{$K$-theoretic Vafa-Witten invariants.} Putting $L=\O_S$ in Conjectures \ref{conj2} and \ref{conj3}, we obtain the conjectural formulae for $K$-theoretic Vafa-Witten invariants of \cite[Rem.~1.3, 1.7]{GK3}. 
\end{itemize}
In \cite[Appendix]{GK1}, the first named author and Nakajima conjectured a formula interpolating between Donaldson invariants and virtual Euler numbers of $M:=M_S^H(2,c_1,c_2)$. Conjecture \ref{conj2} also implies this formula (Section \ref{appl}).

Using the same strategy as for Conjecture \ref{conj1}, we verify Conjecture \ref{conj2} in many examples. On the other hand, for Conjecture \ref{conj3}, we \emph{prove} the universal dependence by presenting a variation on an argument of T.~Laarakker \cite{Laa2}, which in turn is an application of Gholampour-Thomas's description of the monopole virtual class in terms of nested Hilbert schemes \cite{GT1, GT2}. 

\begin{theorem} \label{thm1}
There exist universal series
$$
C_1(y,q),\ldots, C_6(y,q) \in 1+q \, \Q[y^{\frac{1}{2}}][[q]]
$$ 
with the following property. Let $S$ be a smooth projective surface with $p_g(S)>0$, $b_1(S)=0$, and $L \in \Pic(S)$. Let $H,c_1,c_2$ be chosen such that there are no rank 2 strictly Gieseker $H$-semistable Higgs sheaves on $S$ with Chern classes $c_1,c_2$. Let $N:=N_S^H(2,c_1,c_2)$ and let $\vd$ defined by \eqref{vd}. Then the monopole contribution to $\chi(N, \widehat{\O}^{\vir}_N \otimes \mu(L))$ equals the coefficient of $(-x)^{\vd}$ of 
\begin{align*}
&C_1(y,x^4)^{\chi(\O_S)} \, C_2(y,x^4)^{K_S^2} \, C_3(y,x^4)^{L^2} \, C_4(y,x^4)^{LK_S} \\
&\cdot \sum_{a\in H^2(S,\Z)} \delta_{c_1,K_S-a} \, \SW(a) \, \ell_a \, C_5(y,x^4)^{aK_S} \, C_6(y,x^4)^{aL},
\end{align*}
where $\ell_a := x^{a K_S - K_S^2 - 3 \chi(\O_S)} (y^{\frac{1}{2}} + y^{-\frac{1}{2}})^{a K_S - K_S^2 - \chi(\O_S)} y^{\frac{1}{2} L(a-K_S)}$.
 \end{theorem}

For $L=\O_S$ this was proved in \cite{Laa2} (actually, for $L=\O_S$, the analog of this theorem is proved in any rank \cite{Laa2}). Universality on the instanton branch is still open.
The universal series $C_i$ can be expressed in terms of intersection numbers on products of Hilbert schemes of points on surfaces. Again, these intersection numbers are determined on $\PP^2$ and $\PP^1 \times \PP^1$, where we calculate using localization. This way, we determine $C_i \mod q^{15}$ and we find a match with Conjecture \ref{conj2} (Section \ref{mon}). 

For $L = \O_S$, physicists \cite{VW} predict that the instanton and monopole generating functions of Conjectures \ref{conj2} and \ref{conj3} get swapped under the S-duality transformation $\tau \rightarrow -1/\tau$, where $q = \exp(2 \pi i \tau)$. See \cite{JK} for a recent proof of S-duality for K3 surfaces (for any prime rank). For general $L$, the connection between instanton and monopole contribution is less clear. However, the series depending on $L^2$ are related by $x \mapsto x^4$, $y \mapsto y^2$, $L \mapsto L^{\otimes 2}$ (and similarly for any rank in Section \ref{sec:higher}).

\subsection{K3 surfaces}

By adapting an argument from \cite{GNY2} combined with a new formula for twisted elliptic genera of Hilbert schemes of points on surfaces, the first named author proves Conjecture \ref{conj2} for K3 surfaces in \cite{Got2}. By adapting an argument of \cite{Laa2} combined with the above-mentioned formula for twisted elliptic genera of Hilbert schemes of points on surfaces, we prove the following (where the formula for $C_1$ was previously determined in \cite{Tho, Laa2}):
\begin{theorem} \label{thm2}
The universal functions $C_1(y,q), C_3(y,q)$ are given by
\begin{align*}
C_1(y,q) &= \prod_{n=1}^{\infty} \frac{1}{(1-q^{2n})^{10}(1-q^{2n}y^2)(1-q^{2n}y^{-2})}, \\
C_3(y,q) &= \prod_{n=1}^{\infty} \left(\frac{(1-q^{2n})^2}{(1-q^{2n}y^2)(1-q^{2n}y^{-2})}\right)^{2n^2}.
\end{align*}
\end{theorem}
In particular, Conjectures \ref{conj2} and \ref{conj3} hold for K3 surfaces.\footnote{The statement that Conjecture \ref{conj3} holds for K3 surfaces has less content than initially meets the eye. On a K3 surface, $\delta_{c_1,a}$ is only non-zero when $c_1$ is even. Assuming $\gcd(2,c_1H,\frac{1}{2}c_1^2-c_2)=1$, which guarantees ``stable=semistable'', implies $c_2$ is odd. Hence the coefficient of $(-1)x^{\vd}$ of the conjectured expression is always zero. Indeed ``stable=semistable'' implies that the monopole branch is empty \cite[Prop.~7.4]{TT1}.} \\

\noindent \textbf{Acknowledgements.} We thank T.~Laarakker and R.~P.~Thomas for useful discussions. We thank Laarakker for providing his SAGE code for calculating the monopole contribution to Vafa-Witten invariants \cite{Laa2}, which allowed us to gather evidence for Conjecture \ref{conj3}. We also thank the anonymous referee for a through reading of the manuscript and useful suggestions.

\section{Instanton contribution and Donaldson invariants} \label{inst}

In this section we gather evidence for Conjectures \ref{conj1} and \ref{conj2} as follows:
\begin{itemize}
\item \textbf{Reduction to Donaldson invariants.} Express the invariants of Conjectures \ref{conj1} and \ref{conj2} in terms of Donaldson invariants of $S$.
\item \textbf{Reduction to Hilbert schemes.} Use Mochizuki's formula \cite[Thm.~1.4.6]{Moc} to express these invariants as intersection numbers on Hilbert schemes of points on $S$.
\item \textbf{Reduction to toric surfaces.} Show that the intersection numbers of the previous step are determined on $S = \PP^2$ and $\PP^1 \times \PP^1$, where they can be calculated using localization.
\end{itemize}
The final step allows us to calculate the invariants of Conjectures \ref{conj1} and \ref{conj2} and compare to our conjectured formulae. This strategy has been used by the first and second named author in the determination of the instanton contribution to rank 2 and 3 Vafa-Witten invariants and various refinements thereof \cite{GK1,GK2,GK3}. Mochizuki's formula was also used by the first named author and Nakajima-Yoshioka in their proof of the Witten conjecture for algebraic surfaces, which expresses (primary, rank 2) Donaldson invariants in terms of Seiberg-Witten invariants \cite{GNY3}.

\subsection{Donaldson invariants} \label{Doninv} 

Let $S$ be a smooth projective complex surface such that $b_1(S)=0$. Let $H$ be a polarization on $S$ and let $M:=M_S^H(r,c_1,c_2)$.\footnote{In this paragraph, $r>0$ is arbitrary and we do not require $p_g(S)>0$.} We assume there exist no rank $r$ strictly Gieseker $H$-semistable sheaves on $S$ with Chern classes $c_1, c_2$. For the moment, we also assume there exists a universal family $\EE$ on $M \times S$, though we get rid of this assumption in Remark \ref{dropunivfam}. For any $\alpha \in H^*(S,\Q)$ and $k \geq 0$, define $\mu(\alpha) \in H^*(M,\Q)$ as in \eqref{muinsert} and 
\begin{align*}
\tau_k(\alpha) := \ch_{k+2}(\EE) / \mathrm{PD}(\alpha) \in H^*(M,\Q).
\end{align*}
We refer to $\tau_k(\alpha)$ as a descendent insertion and call it primary when $k=0$. As mentioned in the introduction, if $L \in \Pic(S)$ satisfies $c_1(L)c_1 \in 2 \Z$, then there exists a line bundle on $M$, denoted by $\mu(L)$ and called ``a Donaldson line bundle'', whose class in cohomology is \eqref{muinsert} for $\alpha = c_1(L)$.

Consider the $K$-group $K^0(M)$ generated by locally free sheaves on $M$. For any rank $r$ vector bundle on $M$, define
$$
\Lambda_y V := \sum_{i=0}^{r} [\Lambda^i V] y^i \in K^0(M)[y], \qquad \Sym_y V := \sum_{i=0}^{\infty} [\Sym^i V] y^i \in K^0(M)[[y]].
$$
These expressions can be extended to complexes in $K^0(M)$ by setting $\Lambda_y(-V) = \Sym_{-y} V$ and $\Sym_y(-V) = \Lambda_{-y} V$. For any complex $E \in K^0(M)$, we define
\begin{equation} \label{defT}
\mathsf{X}_y(E) := \ch(\Lambda_y E^{\vee}) \, \td(E).
\end{equation}
Since $\Lambda_y(E \oplus F) = \Lambda_y E \otimes \Lambda_y F$, we obtain
$$
\mathsf{X}_y(E \oplus F) = \mathsf{X}_y(E) \, \mathsf{X}_y(F).
$$
Furthermore, for any $L \in \Pic(M)$ 
$$
\mathsf{X}_y(L) = \frac{L(1+ye^{-L})}{1-e^{-L}}.
$$
\begin{lemma} \label{redtoDon}
Let $S,H,r,c_1,c_2$ and $M:=M_S^H(r,c_1,c_2)$ be as above. Let $L \in \Pic(S)$. Then there exists a polynomial expression $P(\EE)$ in $y$ and certain descendent insertions $\tau_k(\alpha)$ and $\mu(c_1(L))$ such that
$$
\chi^{\vir}_y(M, \mu(L)) = \int_{[M]^{\vir}} \mathsf{X}_{y}(T_M^{\vir}) \, e^{\mu(c_1(L))} =  \int_{[M]^{\vir}} P(\E).
$$
\end{lemma}
\begin{proof}
The first equality is the virtual Hirzebruch-Riemann-Roch theorem \cite[Cor.~3.4]{FG} (or the definition of our invariants when the Donaldson line bundle $\mu(L)$ does not exist on $M$). The second equality was proved for $L = \O_S$ in \cite[Prop.~2.1]{GK1} by applying Grothendieck-Riemann-Roch and the K\"unneth formula to
$$
\ch(T_M^\vir) = \ch(R \hom_{\pi_M}(\EE,\EE)_0[1]).
$$
The argument for any $L$ is the same with $P(\E)$ now involving $\mu(c_1(L))$.
\end{proof}

\subsection{Mochizuki's formula} 

We recall Mochizuki's formula \cite[Thm.~1.4.6]{Moc}. 

Let $S^{[n]}$ be the Hilbert scheme of $n$ points on $S$. On $S^{[n_1]} \times S^{[n_2]} \times S$ we have (pull-backs of) the universal ideal sheaves $\I_1$ and $\I_2$ from both factors. For any $M \in \Pic(S)$, on $S^{[n_1]} \times S^{[n_2]}$ we have (pull-backs of) the tautological bundles $M^{[n_1]}$ and $M^{[n_2]}$ from both factors.
We endow $S^{[n_1]} \times S^{[n_2]}$ with the trivial $\C^*$-action and denote the positive generator of the character group of $\C^*$ by $\s$. Define $s := c_1^{\C^*}(\s)$, then 
$$
H^*(B\C^*,\Q) = H^*_{\C^*}(\pt,\Q) \cong \Q[s].
$$

Fix $L \in \Pic(S)$ and let $P(\E)$ be any polynomial in $\mu(c_1(L))$ and descendent insertions $\tau_{k}(\alpha)$. We assume $P(\E)$ arises from a polynomial expression in $\mu(c_1(L))$ and the Chern classes of $T^{\vir}_{M}$ (e.g.~such as in Proposition \ref{redtoDon}). Let $A^1(S)$ be the Chow group of codimension 1 cycles up to linear equivalence, then for any $a_1, a_2 \in A^1(S)$ and $n_1, n_2 >0$, we define (following Mochizuki) 
\begin{align} 
\begin{split} \label{Psi}
&\Psi(L,a_1,a_2,n_1,n_2) := \\
&\Coeff_{s^0} \Bigg( \frac{P(\I_1(a_1) \otimes \s^{-1} \oplus \I_2(a_2) \otimes \s)}{Q(\I_1(a_1) \otimes \s^{-1}, \I_2(a_2) \otimes \s)} \frac{e(\O(a_1)^{[n_1]}) \, e(\O(a_2)^{[n_2]} \otimes \s^2)}{(2s)^{n_1+n_2 - \chi(\O_S)}} \Bigg).
\end{split}
\end{align}
We explain the notation. Here $\I_i(a_i)$ stands for $\I_i \otimes \pi_S^* \O(a_i)$ considered as a sheaf on $S^{[n_1]} \times S^{[n_2]} \times S$ pulled back along projection to $S^{[n_i]} \times S$. Similarly $\O(a_i)^{[n_i]}$ is viewed as a vector bundle on $S^{[n_1]} \times S^{[n_2]}$ pulled back along projection to $S^{[n_i]}$. Since $S^{[n_1]} \times S^{[n_2]}$ has a trivial $\C^*$-action, we can view $\O(a_i)^{[n_i]}$ as endowed with the trivial $\C^*$-equivariant structure. Moreover
$$
\O(a_2)^{[n_2]} \otimes \s^2
$$
denotes $\O(a_2)^{[n_2]}$ with $\C^*$-equivariant structure given by tensoring with character $\s^2$. Similarly, we endow $S^{[n_1]} \times S^{[n_2]} \times S$ with trivial $\C^*$-action, give $\I_i(a_i)$ the trivial $\C^*$-equivariant structure, and denote by
$$
\I_1(a_1) \otimes \s, \qquad \I_2(a_2) \otimes \s^{-1}
$$
the $\C^*$-equivariant sheaves obtained by tensoring with the characters $\s$ and $\s^{-1}$ respectively. We denote the $\C^*$-equivariant Euler class by $e(\cdot)$. Moreover, $P(\cdot)$ stands for the expression obtained from $P(\E)$ by formally replacing $\E$ by $\cdot$ and all Chern classes by $\C^*$-equivariant Chern classes.\footnote{The replacement of $\E$ by $\I_1(a_1) \otimes \s^{-1} \oplus \I_2(a_2) \otimes \s$ comes from Mochizuki's wall-crossing on the master space \cite{Moc}.} For any $\C^*$-equivariant sheaves $E_1$, $E_2$ on $S^{[n_1]} \times S^{[n_2]} \times S$ flat over $S^{[n_1]} \times S^{[n_2]}$
$$
Q(E_1,E_2) :=e(- R\hom_\pi(E_1,E_2) -  R\hom_\pi(E_2,E_1)), 
$$
where $\pi : S^{[n_1]} \times S^{[n_2]} \times S \rightarrow S^{[n_1]} \times S^{[n_2]}$ denotes projection. Finally $\Coeff_{s^0}(\cdot)$ takes the coefficient of $s^0$. We define 
$
\widetilde{\Psi}(L,a_1,a_2,n_1,n_2,s)
$
by expression \eqref{Psi} \emph{without} $\Coeff_{s^0}(\cdot)$. Let $c_1,c_2$ be a choice of Chern classes. For any decomposition $c_1 = a_1 + a_2$, we define (again following Mochizuki)
\begin{equation} \label{cA}
\cA(L,a_1,a_2,c_2) := \sum_{n_1 + n_2 = c_2 - a_1 a_2} \int_{S^{[n_1]} \times S^{[n_2]}} \Psi(L,a_1,a_2,n_1,n_2).
\end{equation}
Let $\widetilde{\cA}(L,a_1,a_2,c_2,s)$ be defined by the same expression with $\Psi$ replaced by $\widetilde{\Psi}$. \\

\begin{theorem}[Mochizuki] \label{mocthm}
Let $S$ be a smooth projective surface satisfying $b_1(S) = 0$, $p_g(S) >0$, and let $L \in \Pic(S)$. Let $H, c_1,c_2$ be chosen such that there are no rank 2 strictly Gieseker $H$-semistable sheaves on $S$ with Chern classes $c_1,c_2$ and such that a universal sheaf $\E$ on $M_{S}^{H}(2,c_1,c_2) \times S$ exists. Assume the following hold:
\begin{enumerate}
\item[(i)] $\chi(\ch) > 0$, where $\chi(\ch) := \int_S \ch \cdot \td(S)$ and $\ch = (2,c_1,\frac{1}{2}c_1^2 - c_2)$.
\item[(ii)] $p_{\ch} > p_{K_S}$, where $p_{\ch} = \chi(e^{mH} \cdot \ch)/2$ and $p_{K_S} = \chi(e^{mH} \cdot e^{K_S})$ are the reduced Hilbert polynomials of $\ch$ and $K_S$.
\item[(iii)] For all SW basic classes $a_1$ satisfying $a_1 H \leq (c_1 -a_1) H$ the inequality is strict. 
\end{enumerate}
Let $P(\E)$ be any polynomial in $\mu(c_1(L))$ and descendent insertions arising from a polynomial in $\mu(c_1(L))$ and Chern classes of $T^\vir_M$ (e.g.~as in Prop.~\ref{redtoDon}). Then
\begin{equation} \label{mocform}
\int_{[M_{S}^{H}(2,c_1,c_2)]^{\vir}} P(\E) = -2^{1-\chi(\ch)} \sum_{{\scriptsize{\begin{array}{c} c_1 = a_1 + a_2 \\ a_1 H < a_2 H \end{array}}}} \SW(a_1) \, \cA(L,a_1,a_2,c_2).
\end{equation}
\end{theorem}

\begin{remark} \label{dropunivfam}
Assuming the existence of a universal sheaf $\E$ on $M \times S$, where $M:=M_{S}^{H}(2,c_1,c_2)$, is unnecessary. As remarked in the introduction, $T_M^{\vir}$ and $\mu(c_1(L))$ always exist, so the left-hand side of Mochizuki's formula is always defined. Mochizuki \cite{Moc} works over the Deligne-Mumford stack of oriented sheaves, which has a universal sheaf. This can be used to show that global existence of $\E$ on $M \times S$ can be dropped from the assumptions. In fact, when working on the stack, $P$ can be \emph{any} polynomial in descendent insertions defined using the universal sheaf on the stack. Also, since Mochizuki works on the stack, his formula and our version differ by a factor 2.
\end{remark}

\begin{remark} \label{assumpmocthm}
Conjecturally, assumptions (ii) and (iii) can be dropped from Theorem \ref{mocthm} \cite{GNY3, GK1, GK2, GK3}. Moreover, also conjecturally, in the sum in Mochizuki's formula the inequality $a_1 H < a_2 H$ can be dropped. Assumption (i) is necessary. 
\end{remark}

Suppose the assumptions of Theorem \ref{mocthm} are satisfied. Combining with Lemma \ref{redtoDon}, we find that $y^{-\frac{\vd}{2}} \chi^{\vir}_{-y}(M, \mu(L))$ is given by \eqref{mocform} with
\begin{equation} \label{choiceP}
P(\E) = y^{-\frac{\vd}{2}} \mathsf{X}_{-y}(-R\hom_\pi(\E,\E)_0) \, e^{\mu(c_1(L))},
\end{equation}
where $\E$ is replaced by
$$
\I_1(a_1) \otimes \s^{-1} \oplus \I_2(a_2) \otimes \s.
$$
We note that the rank of
\begin{equation*}  
-R\hom_\pi(\I_1(a_1) \otimes \s^{-1} \oplus \I_2(a_2) \otimes \s,\I_1(a_1) \otimes \s^{-1} \oplus \I_2(a_2) \otimes \s)_0
\end{equation*} 
equals the rank of $T^{\vir}_M = -R\hom_\pi(\E,\E)_0$.

\subsection{Universal series}

In this paragraph, $S$ is \emph{any} smooth projective surface, so we allow $p_g(S) = 0$. We want to study the intersection numbers \eqref{cA} with $P(\E)$ given by \eqref{choiceP}. Let $\mathsf{X}^{\C^*}_y(\cdot)$ denote the same expression as in \eqref{defT}, but with Chern character and Todd class replaced by \emph{$\C^*$-equivariant} Chern character and Todd class (recall that we endow $S^{[n_1]} \times S^{[n_2]}$ with trivial $\C^*$-action). Define
\begin{equation*} 
f(s,y):=y^{-\frac{1}{2}} \, \mathsf{X}^{\C^*}_{-y}(\s^2) = y^{-\frac{1}{2}} \frac{2s(1-ye^{-2s})}{1-e^{-2s}}
\end{equation*}
where the second equality follows from the properties listed in Section \ref{Doninv}. We write $\chi(a):=\chi(\O_S(a))$ for any $a \in A^1(S)$. For \emph{any} $L,a,c_1 \in A^1(S)$, we define
\begin{align*}
\sfZ_S^{\inst}(L,a,c_1,s,y,q) := &(2s)^{-\chi(\O_S)} \Big( \frac{2s}{f(s,y)} \Big)^{-\chi(c_1-2a)} \Big( \frac{-2s}{f(-s,y)} \Big)^{-\chi(2a-c_1)} e^{(c_1-2a)Ls} \\
&\cdot \sum_{n_1,n_2} q^{n_1+n_2} \int_{S^{[n_1]} \times S^{[n_2]}} \widetilde{\Psi}(L,a,c_1-a,n_1,n_2,s).
\end{align*}
The first line of this expression  is just a normalization factor, so
$$
\sfZ_S^{\inst}(L,a,c_1,s,y,q) \in 1 + q \, \Q[y^{\pm \frac{1}{2}}] (\!(s)\!) [[q]].
$$
We note that the definition of $\sfZ_S^{\inst}(L,a,c_1,s,y,q)$ makes sense for any \emph{possibly disconnected} smooth projective surface $S$ and $L,a,c_1 \in A^1(S)$.

\begin{lemma} \label{mult}
Let $S = S' \sqcup S''$, where $S', S''$ are (possible disconnected) smooth projective surfaces. Let $L,a,c_1 \in A^1(S)$ and define $L':=L|_{S'}$, $a':=a|_{S'}$, $c_1':=c_1|_{S'}$, $L'':=L|_{S''}$, $a'':=a|_{S''}$, and $c_1'':=c_1|_{S''}$. Then
$$
\sfZ_S^{\inst}(L,a,c_1,s,y,q) = \sfZ_{S'}^{\inst}(L',a',c_1',s,y,q) \, \sfZ_{S''}^{\inst}(L'',a'',c_1'',s,y,q).
$$
\end{lemma}
\begin{proof}
The case $L = \O_S$ was established in \cite[Prop.~3.3]{GK1}. The only new feature of the present case is the following. 

Define $S_2 = S \sqcup S$ . As shown in \cite[Prop.~3.3]{GK1}, the integrals over $S^{[n_1]} \times S^{[n_2]}$ occuring in the coefficients of $\sfZ_S^{\inst}(L,a,c_1,s,y,q)$ can be written as integrals on $S_2^{[n]}$ by using the decomposition
$$
S_2^{[n]} = \bigsqcup_{n_1+n_2=n} S^{[n_1]} \times S^{[n_2]}. 
$$
Since $S = S' \sqcup S^{\prime \prime}$, we have a further decomposition
$$
S^{[n_1]} \times S^{[n_2]} = \bigsqcup_{l_1+l_2=n_1, m_1+m_2=n_2} S^{ \prime [l_1]} \times S^{ \prime \prime [l_2]} \times S^{ \prime [m_1]} \times S^{ \prime \prime [m_2]}.
$$
Then the insertion $e^{\mu(c_1(L))}$ restricted to $S^{ \prime [l_1]} \times S^{ \prime \prime [l_2]} \times S^{ \prime [m_1]} \times S^{ \prime \prime [m_2]}$ equals
\begin{equation*}
p^{\prime *}  e^{\mu(c_1(L'))} p^{\prime \prime *} e^{\mu(c_1(L^{\prime \prime}))},
\end{equation*}
where $p', p''$ are the projections in the diagram
\begin{displaymath}
\xymatrix
{
& S^{ \prime [l_1]} \times S^{ \prime \prime [l_2]} \times S^{ \prime [m_1]} \times S^{ \prime \prime [m_2]} \ar[dl]^{p'} \ar[dr]_{p''} & \\
S^{ \prime [l_1]} \times S^{ \prime [m_1]} & & S^{ \prime \prime [l_2]} \times S^{ \prime \prime [m_2]}
}
\end{displaymath}
and $S^{ \prime [l_1]} \times S^{ \prime [m_1]}$ is seen as a connected component of $S_2^{\prime [l_1+m_1]}$ and $S^{ \prime \prime [l_2]} \times S^{ \prime \prime [m_2]}$ as a connected component of $S_2^{\prime \prime [l_2+m_2]}$. The rest of the proof proceeds exactly as in \cite[Prop.~3.3]{GK1}.
\end{proof}

\begin{lemma} \label{A}
There exist universal functions 
$$
A_1(y,q),\ldots , A_{11}(y,q) \in 1+q \, \Q[y^{\pm \frac{1}{2}}][[q]]
$$ 
such that for any smooth projective surface $S$ and $L,a,c_1 \in A^1(S)$ we have
$$
\sfZ_S^{\inst}(L,a,c_1,s,y,q) = A_1^{L^2} A_2^{L a} A_3^{a^2} A_4^{a c_1} A_5^{c_1^2} A_6^{L c_1} A_7^{L K_S} A_8^{a K_S} A_9^{c_1 K_S} A_{10}^{K_S^2} A_{11}^{\chi(\O_S)}.
$$
\end{lemma}
\begin{proof}
By \cite{EGL}, tautological integrals on Hilbert schemes of points on surfaces are universal. We are dealing with integrals over products of Hilbert schemes, which were handled in \cite[Lem.~5.5]{GNY1}. By \cite[Lem.~5.5]{GNY1} (see also \cite[Prop.~3.3]{GK1}), there exists a universal power series
$$
G \in \Q[x_1, \cdots, x_{11}][[q]]
$$
such that for any smooth projective surface $S$ and $L,a,c_1 \in A^1(S)$ we have
\begin{equation} \label{G}
\sfZ_S^{\inst}(L,a,c_1,s,y,q) = e^{G(L^2,La,a^2,ac_1, c_1^2,Lc_1,LK_S,aK_S,c_1K_S,K_S^2,\chi(\O_S))}.
\end{equation}
Here we use the fact that $\sfZ_S^{\inst}(L,a,c_1,s,y,q)$ starts with 1.

We claim that equation \eqref{G} and Lemma \ref{mult} together imply the result. This can be seen as follows (see also \cite[Lem.~5.5]{GNY1}). Choose 11 quadruples $(S^{(i)},L^{(i)},a^{(i)},c_1^{(i)})$ such that the corresponding vectors of Chern numbers
$$
w_i := ((L^{(i)})^2,\ldots,\chi(\O_{S^{(i)}})) \in \Q^{11}
$$
form a $\Q$-basis. Now consider any $(S,L,a,c_1)$. Then we can decompose its vector of Chern numbers $w = (L^2, \ldots, \chi(\O_S))$ as 
$
w = \sum_i n_i w_i,
$
for some $n_i \in \Q$. If all $n_i \in \Z_{\geq 0}$, then Lemma \ref{mult} implies that 
\begin{equation} \label{intermed}
\sfZ_S^{\inst}(L,a,c_1,s,y,q) = \prod_{i=1}^{11} \Big( e^{G(w_i)} \Big)^{n_i}.
\end{equation}
Let $W$ be the matrix with column vectors $w_1, \ldots, w_{11}$ and $M = (m_{ij})$ its inverse. Defining $A_j := \exp(\sum_i m_{ij} G(w_i))$, 
equation \eqref{intermed} implies
$$
\sfZ_S^{\inst}(L,a,c_1,s,y,q) = A_1^{L^2} \cdots A_{11}^{\chi(\O_S)}.
$$
Since the set of vectors $w$ with all $n_i \in \Z_{\geq 0}$ is Zariski dense in $\Q^{11}$, the proposition holds for \emph{any} $(S,L,a,c_1)$.
\end{proof}

Theorem \ref{mocthm} and Lemma \ref{A} at once imply the following result.
\begin{proposition} \label{mainprop}
Let $S$ be a smooth projective surface with $b_1(S) = 0$, $p_g(S) >0$, and $L \in \Pic(S)$. Let $H, c_1,c_2$ be chosen such that there are no rank 2 strictly Gieseker $H$-semistable sheaves on $S$ with Chern classes $c_1,c_2$. Assume the following hold:
\begin{enumerate}
\item[(i)] $\chi(\ch) > 0$, where $\chi(\ch) := \int_S \ch \cdot \td(S)$ and $\ch = (2,c_1,\frac{1}{2}c_1^2 - c_2)$.
\item[(ii)] $p_{\ch} > p_{K_S}$, where $p_{\ch} = \chi(e^{mH} \cdot \ch)/2$ and $p_{K_S} = \chi(e^{mH} \cdot e^{K_S})$ are the reduced Hilbert polynomials of $\ch$ and $K_S$.
\item[(iii)] For all SW basic classes $a$ with $a H \leq (c_1 -a) H$ the inequality is strict. 
\end{enumerate}
Then  $y^{-\frac{\vd}{2}}\chi_{-y}^{\vir}(M_{S}^{H}(2,c_1,c_2),\mu(L))$ is the coefficient of $x^{\vd} s^0$ of 
\begin{align*} 
&-2 \sum_{{\scriptsize{\begin{array}{c} a \in H^2(S,\Z) \\ a H < (c_1-a) H \end{array}}}}  \SW(a) \, A_1(y,2x^4)^{L^2}  \Bigg( e^{2s} A_2(y,2x^4) \Bigg)^{L a} \\
&\cdot \Bigg(2^{-1} \Bigg( \frac{2s}{f(s,y)} \Bigg)^2 \Bigg( \frac{-2s}{f(-s,y)} \Bigg)^2 x^{-4} A_3(y,2x^4) \Bigg)^{a^2} \\
&\cdot \Bigg( 2 \Bigg(\frac{2s}{f(s,y)} \Bigg)^{-2} \Bigg( \frac{-2s}{f(-s,y)} \Bigg)^{-2} x^4 A_4(y,2x^4) \Bigg)^{a c_1} \\
&\cdot \Bigg(2^{-\frac{1}{2}} \Bigg( \frac{2s}{f(s,y)} \Bigg)^{\frac{1}{2}} \Bigg( \frac{-2s}{f(-s,y)} \Bigg)^{\frac{1}{2}} x^{-1} A_5(y,2x^4) \Bigg)^{c_1^2} \\
&\cdot \Bigg( e^{-s} A_6(y,2x^4) \Bigg)^{L c_1} A_7(y,2x^4)^{L K_S} \\
&\cdot \Bigg(  \Bigg( \frac{2s}{f(s,y)} \Bigg) \Bigg( \frac{-2s}{f(-s,y)} \Bigg)^{-1}  A_8(y,2x^4) \Bigg)^{a K_S} \\
&\cdot \Bigg( 2^{\frac{1}{2}} \Bigg( \frac{2s}{f(s,y)} \Bigg)^{-\frac{1}{2}} \Bigg( \frac{-2s}{f(-s,y)} \Bigg)^{\frac{1}{2}}  A_9(y,2x^4) \Bigg)^{c_1 K_S} \\
&\cdot A_{10}(y,2x^4)^{K_S^2} \Bigg( \frac{s}{2} \Bigg( \frac{2s}{f(s,y)} \Bigg) \Bigg( \frac{-2s}{f(-s,y)} \Bigg) x^{-3} A_{11}(y,2x^4) \Bigg)^{\chi(\O_S)}.
\end{align*}
\end{proposition}

\begin{remark} 
By Remark \ref{assumpmocthm}, conjecturally, assumptions (ii) and (iii) in the previous proposition, as well as the inequality $aH < (c_1-a)H$ in the sum, can be dropped. 
\end{remark}

\subsection{Reduction to toric surfaces} \label{toricA}

We now present 11 choices of $(S,L,a,c_1)$ for which the vectors of Chern numbers $(L^2,\ldots,\chi(\O_S))$ are $\Q$-independent:
\begin{align*}
(S,L,a,c_1) = &(\PP^2,\O,\O,\O), \\ 
& (\PP^1 \times \PP^1,\O,\O,\O), \\
&(\PP^2,\O,\O(1),\O(2)), \\ 
&(\PP^2,\O,\O,\O(1)), \\
& (\PP^2,\O,\O(1),\O(3)), \\
& (\PP^1 \times \PP^1,\O,\O(0,1),\O(0,2)), \\
& (\PP^1 \times \PP^1,\O,\O,\O(0,1)), \\
&(\PP^2,\O(1),\O,\O), \\ 
&(\PP^1\times \PP^1,\O(0,1),\O,\O), \\
&(\PP^2,\O(1),\O(1),\O(2)),\\
&(\PP^2,\O(1),\O,\O(1)).
\end{align*}
Each of these surfaces $S$ is toric and hence has an action of $T=\C^* \times \C^*$. Choose $T$-equivariant structures on the line bundles corresponding to $L,a,c_1$. Then we can calculate $\sfZ_S^{\inst}(L,a,c_1,s,y,q)$ by Atiyah-Bott localization. More precisely, consider one of the intersection numbers
$$
\int_{S^{[n_1]} \times S^{[n_2]}} \widetilde{\Psi}(L,a,c_1-a,n_1,n_2,s)
$$
appearing in the definition of $\sfZ_S^{\inst}(L,a,c_1,s,y,q)$. The action of $T$ lifts to $S^{[n_1]} \times S^{[n_2]}$ and its fixed locus is indexed by pairs
$$
\left(\{\lambda^{(\sigma)}\}_{\sigma=1}^{e(S)},\{\mu^{(\sigma)}\}_{\sigma=1}^{e(S)} \right),
$$ 
where each $\lambda^{(\sigma)} = (\lambda^{(\sigma)}_1 \geq \lambda^{(\sigma)}_2 \geq  \ldots)$ and $\mu^{(\sigma)} = (\mu^{(\sigma)}_1 \geq \mu^{(\sigma)}_2 \geq  \ldots)$ are partitions such that
$$
\sum_\sigma |\lambda^{(\sigma)}| = \sum_{\sigma,i} \lambda^{(\sigma)}_i = n_1, \qquad \sum_\sigma |\mu^{(\sigma)}| = \sum_{\sigma, i} \mu^{(\sigma)}_i = n_2.
$$
The Euler number $e(S)$ equals the number of torus fixed points $p_\sigma$ of $S$ and each partition $\lambda^{(\sigma)}$, $\mu^{(\sigma)}$ corresponds (in the usual way) to a monomial ideal on the maximal $T$-invariant affine open subset $\C^2 \cong U_\sigma \subset S$ containing $p_\sigma$. E.g.~see \cite{GK1, GK2} for more details. 

For any pair $(\{\lambda^{(\sigma)}\}_\sigma,\{\mu^{(\sigma)}\}_\sigma)$ corresponding to 0-dimensional $T$-fixed subschemes $(Z,W) \in S^{[n_1]} \times S^{[n_2]}$, we are interested in the restriction
\begin{equation} \label{psirestr}
\widetilde{\Psi}(L,a,c_1-a,n_1,n_2,s) \Big|_{(Z,W)}.
\end{equation}
Let $\widetilde{T} := T \times \C^*$, where $\C^*$ is the torus acting trivially on $S^{[n_1]} \times S^{[n_2]}$ (as in Mochizuki's formula). Denote by $\mathfrak{t}_1, \mathfrak{t}_2, \mathfrak{s}$ positive primitive generators of the character group of each factor of $\widetilde{T}$. Then the $\widetilde{T}$-equivariant $K$-group of a point is given by the following ring of Laurent polynomials
$$
K^0_{\widetilde{T}}(\pt) \cong \Z[\mathfrak{t}_1^{\pm}, \mathfrak{t}_2^{\pm}, \mathfrak{s}^{\pm}].
$$ 
In order to calculate \eqref{psirestr} in terms of $\epsilon_1 := c_1^{\widetilde{T}}(\mathfrak{t}_1)$, $\epsilon_2 := c_1^{\widetilde{T}}(\mathfrak{t}_2)$, and $s := c_1^{\widetilde{T}}(\s)$, we must determine the classes of the following complexes in $K^0_{\widetilde{T}}(\pt)$
\begin{align*}
&H^0(\O_Z(a)), \quad H^0(\O_W(c_1-a)), \\
&R\Hom_S(\O_Z,\O_Z), \quad R\Hom_S(\O_W,\O_W),  \\
&R\Hom_S(\O_Z,\O_W(c_1-2a)\otimes \s^2), \quad R\Hom_S(\O_W(c_1-2a) \otimes \s^2,\O_Z), 
\end{align*}
where $I_Z, I_W \subset \O_S$ are the ideal sheaves of $Z,W$.
The expressions in the first line follow at once from the $\widetilde{T}$-representations of $Z, W$ in terms of the partitions $\lambda^{(\sigma)}, \mu^{(\sigma)}$. The expressions in lines two and three can be calculated by using a $T$-equivariant resolution of $I_Z, I_W$. For explicit formulae, see \cite[Prop.~4.1]{GK1}. Finally, $\mu(L)$ leads to the insertion
\begin{align*}
\pi_* \left( c^{\widetilde{T}}_1(L) \cdot ( \ch^{\widetilde{T}}_2(\O_Z) + \ch^{\widetilde{T}}_2(\O_W)) \cap [S]  \right) = \sum_{\sigma=1}^{e(S)} a_\sigma \cdot \left( |\lambda^{(\sigma)}| + |\mu^{(\sigma)}| \right),
\end{align*}
where $\pi_* : K^0_{\widetilde{T}}(S) \rightarrow K^0_{\widetilde{T}}(\pt)$ denotes equivariant push-forward and $a_\sigma$ is the character corresponding to $L|_{U_\sigma}$. 

The calculation of  $\sfZ_S^{\inst}$ for each of the 11 cases above is now a purely combinatorial problem, which we implemented in Pari/GP. We determined the universal series $A_1,\ldots , A_{11}$ of Proposition \ref{A} to the following orders:
\begin{itemize}
\item For $A_{1}(1,q),\ldots,A_{11}(1,q)$, we computed the coefficients of $s^{l-3n} q^n$ for all $n \leq 10$, $l\le 49$. (Recall: $A_i(1,q), A_i(y,q)$ are \emph{Laurent} series in $s$.)
\item For $A_{1}(y,q),\ldots,A_{11}(y,q)$, we computed the coefficients of $s^{l-5n} y^m q^n$ for all $n \leq 6$, $m \leq 9$, $l\le 30$.
\end{itemize}

\subsection{Verifications}

We verified Conjecture \ref{conj1} in the following cases.\footnote{In this list, $\pi : S \rightarrow S'$ always denotes the blow-up in a point and the exceptional divisor is written as $E$ (or $E_1, E_2$ in the case of a blow-up in two dinstinct points).} In each case, we fix $S,c_1,c_2$ as indicated and we choose $H$ such that the assumptions of Proposition \ref{mainprop} are satisfied. We use the explicit expansions of $A_{1}(1,q),\ldots,A_{11}(1,q)$ determined in the previous section by localization calculations on $\PP^2$ and $\PP^1 \times \PP^1$ using the computer program Pari/GP.
\begin{enumerate}
\item $S$ is a K3 surface, $c_1$ such that $c_1^2=0,2,\ldots,20$, and $\vd < 14$.
\item $S$ is the blow-up of a K3 surface in a point, $c_1=\pi^*C+\epsilon E$ such that $C^2=-4,-2,\ldots,10$, $\epsilon=-2,-1,\ldots,2$, and $\vd < 15$.
\item $S$ is the blow-up of a K3 surface in two distinct points, $c_1=\pi^*C+ e_1 E_1+ e_2 E_2$ such that $C^2=-2,0,\ldots,6$, $e_1,e_2=0,1$, and $\vd < 13$.
\item $S \rightarrow \PP^1$ is an elliptic surface of type $E(N)$,\footnote{I.e.~an elliptic surface $S \rightarrow \PP^1$ with section, $12N$ rational 1-nodal fibres, and no other singular fibres.} $N=\chi(\O_S)=3,4, \ldots, 7$,
$c_1=mB+nF$ where $B$ is the class of a section, $F$ is the class of a fibre, 
$m=-1,0,1,2$, $n=-2,-1,\ldots,5$, and $\vd < 12$.
\item $S$ is the blow-up of an elliptic surface of type $E(3)$ in a point, $c_1=\pi^*C+\epsilon E$ such that $CK_{S}=-1,0,\ldots,4$, $C^2=-4,-3,\ldots,10$, $\epsilon=0,1$, and $\vd < 12$.
\item $S$ is a minimal general type surface with $b_1(S)=0$, $\chi(\O_S)=2$, $K_S^2=1$ \cite{Kyn}, $c_1$ such that $c_1\cdot K_S=0,1$, $c_1^2=-2,-1,\ldots,11$, and $\vd < 12$.
\item $S$ is a double cover of $\PP^2$ branched along a smooth octic, 
$c_1$ such that $c_1 \cdot K_S=0,1,\ldots, 10$, $c_1^2=0,1,\ldots, 30$, and $\vd < 12$. 
\item $S$ is the blow-up of a surface $S'$ as in (7) in a point, $c_1=\pi^* C+\epsilon E$ such that $CK_{S}=-2,-1,\ldots,2$, $C^2=-2,-1,\ldots,8$, $\epsilon=0,1$, and $\vd < 11$.
\item $S$ is a very general smooth quintic in $\PP^3$ (then 
$\Pic(S) = \Z[H]$), $c_1 =2H$ and $\vd < 8$, or $c_1=3H$ and $\vd < 7$. 
\end{enumerate}

Assuming the strong form of Mochizuki's formula holds (Remark \ref{assumpmocthm}), we also verified Conjecture \ref{conj1} in the following cases:
\begin{enumerate}
\item[(10)] $S$ is a smooth quintic surface in $\PP^3$, $c_1$ such that $c_1 \cdot K_S=0,1, \ldots,  25$, $c_1^2=-4,-3, \ldots, 20$, and $\vd < 11$. 
\item[(11)] $S$ is the blow-up of a quintic in $\PP^3$ in a point, $c_1=\pi^*C+\epsilon E$ such that $C K_S=-5,-4,\ldots,5$, $C^2=-4,-3,\ldots,8$, $\epsilon=0,1$, and $\vd < 10$.
\end{enumerate}

Applying the same method and using our explicit expansions of $A_{1}(y,q)$, $\ldots$, $A_{11}(y,q)$ from the previous section, we verified Conjecture \ref{conj2} in the following cases:
\begin{enumerate} 
\item $S$ is a K3 surface, $c_1$ such that $c_1^2=0,2,\ldots,14$, and $\vd < 11$.
\item $S$ is the blow-up of a K3 surface in a point, $c_1=\pi^*C+\epsilon E$ such that $C^2=-4,-2,\ldots,14$, $\epsilon=-2,-1,\ldots,2$, and $\vd <10$.
\item $S$ is the blow-up of a K3 surface in two distinct points, $c_1=\pi^*C+e_1 E_1+e_2 E_2$ such that $C^2=-2,0,\ldots,6$, $e_1,e_2=0,1$, and $\vd < 10$.
\item $S$ is an elliptic surface of type $E(N)$ with $N=3,4,5$, $c_1=mB+nF$ with $m=-1,0,1,2$, $n=-2,-1,\ldots,10$, and $\vd <9$.
\item $S$ is the blow-up of an elliptic surface of type $E(3)$ in a point, $c_1=\pi^* C+\epsilon E$ such that $CK_{S}=-1,0,\ldots,4$, $C^2=-16,-15,\ldots,0$, $\epsilon=0,1$, and
$\vd <9$.
\item $S$ is the double cover of $\PP^2$ branched along a smooth octic, $c_1$ such that $c_1 \cdot K_S=-2,-1,\ldots, 2$, $c_1^2=-16, -15,\ldots, -6$, and $\vd < 9$. 
\item $S$ is the blow-up of $S'$ as in (6) in a point, $c_1=\pi^* C+\epsilon E$ such that $CK_S=-2,-1,\ldots,2$, $C^2=-16, -15\ldots,8$, $\epsilon=0,1$, and $\vd < 7$.
\end{enumerate}

Assuming the strong form of Mochizuki's formula holds (Remark \ref{assumpmocthm}), we also verified Conjecture \ref{conj2} in the following cases:
\begin{enumerate}
\item[(8)] $S$ is a smooth quintic in $\PP^3$, $c_1$ such that $c_1 \cdot K_S=2,3,\ldots,6$, $c_1^2=-16,-15,\ldots,-3$, and $\vd < 7$.
\item[(9)] $S$ is the blow-up of a smooth quintic in $\PP^3$ in a point, $c_1=\pi^*C+\epsilon E$ such that $CK_S=0$, $C^2=-23,-22,\ldots,-14$, $\epsilon=0,1$, and $\vd < 4$.
\end{enumerate}

\section{Monopole contribution and nested Hilbert schemes} \label{mon}

In this section, we study the contribution of the monopole branch to the invariants $\chi(N, \widehat{\O}^{\vir}_N \otimes \mu(L))$ defined in the introduction. We prove that it is determined by universal series $C_1, \ldots, C_6$ as stated in Theorem \ref{thm1}. Moreover, we express these universal functions in terms of integrals over products over Hilbert schemes of points on $S$. Much like in the previous section, these integrals are determined by their value on $\PP^2$ and $\PP^1 \times \PP^1$, where we calculate them, modulo $q^{15}$, by localization.

The methods of this section are a variation on Laarakker's work \cite{Laa2}, which in turn relies on Gholampour-Thomas's work \cite{GT1, GT2}. For $L =\O_S$ and $r=2$, Theorem \ref{thm1} was previously proved in \cite{Laa2} (in fact, for $L=\O_S$, he proved the analog of Theorem \ref{thm1} in any rank). Then $\chi(N, \widehat{\O}^{\vir}_N)$ are the rank 2 $K$-theoretic Vafa-Witten invariants defined by Thomas \cite{Tho} and determined by the universal series $C_1, C_2, C_5$. Closed formulae for these universal series were conjectured in \cite{GK3} (refining Vafa-Witten's original formula \cite[(5.38)]{VW}) and subsequently verified in \cite{Laa2} up to the following orders:
\begin{align} \label{Cpart1}
\begin{split}
C_1(y,q) &= \prod_{n=1}^{\infty} \frac{1}{(1-q^{2n})^{10}(1-q^{2n}y^2)(1-q^{2n}y^{-2})} \mod q^{15} \\
C_2(y,q) &= (y^{\frac{1}{2}} + y^{-\frac{1}{2}}) q^{\frac{1}{4}} \frac{\overline \eta(q)^2}{\theta_2(q,y)} \mod q^{15} \\ 
C_5(y,q) &= \frac{1}{(y^{\frac{1}{2}} + y^{-\frac{1}{2}}) q^{\frac{1}{4}}} \frac{\theta_2(q,y)}{\theta_3(q,y)} \mod q^{15},
\end{split}
\end{align}
where $\overline{\eta}(q), \theta_2(q,y), \theta_3(q,y)$ were introduced in \eqref{thetadef}. The universal power series $C_3, C_4, C_6$ are new. In accordance with Conjecture \ref{conj3}, we show 
\begin{align}
\begin{split}  \label{Cpart2}
C_3(y,q) = \prod_{n=1}^{\infty} &\left(\frac{(1-q^{2n})^2}{(1-q^{2n}y^2)(1-q^{2n}y^{-2})}\right)^{2n^2} \mod q^{15} \\
C_4(y,q) =  \prod_{n=1}^{\infty} &\left( \frac{1- q^{n}y^{-1}}{1- q^{n}y}\right)^n \left(\frac{1-q^{2n}y^{-2}}{1-q^{2n}y^2}\right)^{n}  \left( \frac{1+q^{2n}y^{-1}}{1+q^{2n}y}\right)^{4n}\mod q^{15} \\
C_6(y,q) =  \prod_{n=1}^{\infty} &\left( \frac{1-(-1)^n q^{n}y^{-1}}{1-(-1)^n q^{n}y}\right)^{2n} \left( \frac{1-q^{4n}y^2}{1-q^{4n}y^{-2}}\right)^{4n} \mod q^{15}.
\end{split}
\end{align}

\subsection{Gholampour-Thomas's formula} \label{mon1} 

Let $S$ be a smooth projective surface satisfying $b_1(S) = 0$ and $p_g(S)>0$. Let $r=2$ and $c_1,c_2$ be chosen such that there are no rank 2 strictly Gieseker $H$-semistable Higgs sheaves on $S$ with Chern classes $c_1,c_2$. Let $N := N_{S}^{H}(2,c_1,c_2)$ and let $M^{\mon} \subset N^{\C^*}$ be the monopole branch discussed in the introduction. Gholampour-Thomas \cite{GT1} (see also \cite{GSY1}) prove that the components of $M^{\mon}$ are isomorphic to 
$$
S_{\beta}^{[n_0,n_1]} :=\{ (Z_0,Z_1,C) \, : \, I_{Z_0}(-C) \subset I_{Z_1} \} \subset S^{[n_0]} \times S^{[n_1]} \times |\beta|,
$$
for \emph{certain} (see Remark \ref{ignorestab} below) $n_0,n_1 \geq 0$ and algebraic $\beta \in H^2(S,\Z)$. In particular, such $n_0,n_1,\beta$ satisfy 
\begin{align}
\begin{split} \label{n0n1beta}
&c_1 - \beta +K_S \in 2H^2(S,\Z), \\
&c_2 =  n_0+n_1+\Big( \frac{c_1-\beta+K_S}{2} \Big) \Big( \frac{c_1+\beta-K_S}{2} \Big).
\end{split}
\end{align}
Whenever we have $n_0,n_1,\beta$ satisfying \eqref{n0n1beta}, it is convenient to define
\begin{align*}
L_0 := \frac{c_1-\beta+K_S}{2}, \quad L_1 := \frac{c_1+\beta-K_S}{2}.
\end{align*}
Consider the inclusion
$$
\iota :   S_{\beta}^{[n_0,n_1]} \subset S^{[n_0]} \times S^{[n_1]}  \times |\beta|,
$$
where $|\beta|$ denotes the linear system determined by $\O_S(\beta)$. The universal sheaf $\EE$ on $M^{\mon} \times S$ restricted to the component $S_{\beta}^{[n_0,n_1]} \times S$ is 
 \begin{equation} \label{Emono}
 \EE \cong \I_0 \otimes L_0 \oplus  \I_1 \otimes L_1(1) \otimes \t^{-1},
 \end{equation}
 where $\t$ is a positive primitive character of the trivial $\C^*$-action on $M^{\mon} \times S \subset N^{\C^*} \times S$. Moreover, $\I_0, \I_1$ are the universal ideal sheaves pulled back  from the factors of $S^{[n_0]} \times S^{[n_1]}  \times S \times |\beta|$ (and then along $\iota \times \id_S$ to $S_{\beta}^{[n_0,n_1]}$), $L_0, L_1$ are pulled back from $S$, and $\O(1)$ is pulled back from $|\beta|$. Consider $M^{\mon} \subset N^{\C^*}$ with its $\C^*$-localized perfect obstruction theory \cite{GP}.
 
\begin{theorem}[Gholampour-Thomas] \label{GT}
The class $\iota_* [S_{\beta}^{[n_0,n_1]}]^{\vir}$ is given by
$$
\SW(\beta)  \, e\big( R\Gamma(\beta) \otimes \O - R\hom_{\pi}(\I_0,\I_1(\beta) \big) \in H_{2n_0+2n_1}(S^{[n_0]} \times S^{[n_1]} \times |\beta|),
$$
where $\pi : S^{[n_0]} \times S^{[n_1]} \times S \rightarrow S^{[n_0]} \times S^{[n_1]}$ denotes projection, $e(\cdot) = c_{n_0+n_1}(\cdot)$,\footnote{By Carlsson-Okounkov vanishing, $c_{>n_0+n_1}(R\Gamma(\beta) \otimes \O - R\hom_{\pi}(\I_0,\I_1(\beta)) = 0$ \cite{GT1}.} and the LHS should be interpreted as the image under push-forward along the inclusion $S^{[n_0]} \times S^{[n_1]} \times \{\pt\} \hookrightarrow S^{[n_0]} \times S^{[n_1]} \times |\beta|$ for any point $\pt \in |\beta|$.  
\end{theorem}

\begin{remark} \label{ignorestab}
Not all $n_0,n_1,\beta$ satisfying \eqref{n0n1beta} correspond to spaces $S_{\beta}^{[n_0,n_1]}$ containing Gieseker $H$-stable Higgs sheaves on $S$ with Chern classes $c_1,c_2$. However, such components still have a virtual class given by the formula of Theorem \ref{GT} (induced by realizing $S_\beta^{[n_0,n_1]}$ as an incidence locus inside $S^{[n_0]} \times S^{[n_1]} \times |\beta|$ \cite{GT1}). Laarakker proves that components $S_{\beta}^{[n_0,n_1]}$, which are not part of $M^{\mon}$, satisfy $[S_{\beta}^{[n_0,n_1]}]^{\vir}=0$ \cite{Laa2}. 
This implies we may as well consider \emph{all} $n_0,n_1,\beta$ satisfying \eqref{n0n1beta} and their corresponding spaces $S_{\beta}^{[n_0,n_1]}$.
\end{remark}

\begin{remark} \label{differentvd}
Unlike the instanton branch, it may happen that the monopole branch $M^{\mon}$ of $N_S^H(2,c_1,c_2)^{\C^*}$ has components of different virtual dimension with respect to the $\C^*$-localized perfect obstruction theory. A component $S_{\beta}^{[n_0,n_1]} \subset M^{\mon}$, where $n_0,n_1,\beta$ satisfy \eqref{n0n1beta}, has virtual dimension $n_0 +n_1$. As an example, take $S \rightarrow \PP^2$ a double cover branched over a smooth curve of degree 10, then $K_S = 2L$, where $L \subset S$ is the pull-back of the line from $\PP^2$.  Let $H = L$, $c_1=K_S$, and $c_2 \geq 3$ odd, then $\gcd(2,c_1H,\frac{1}{2}c_1(c_1-K_S) - c_2)=1$, in which case there are no rank 2 strictly Gieseker $H$-semistable Higgs sheaves on $S$ with Chern classes $c_1, c_2$. For $\beta = 0$ and any $0 \leq n_1 \leq n_0$ such that $c_2=n_0+n_1$, we obtain a non-empty component of virtual dimension $c_2$. For $\beta = K_S$ and any $0 \leq n_0 < n_1$ such that $c_2=n_0+n_1+2$, we obtain a non-empty component of virtual dimension $c_2-2$. In both cases, the elements of the component correspond to Gieseker $H$-stable Higgs sheaves. Also note that in this example $\beta=0,K_S$ are the Seiberg-Witten basic classes of $S$.
\end{remark}

Although the virtual dimension of the monopole branch is in general \emph{not} given by \eqref{vd}, we still define
$$
\vd(2,c_1,c_2) := \vd =  4c_2 - c_1^2 - 3 \chi(\O_S)
$$
and use $x^{\vd}$ as the formal variable of our generating series.

\subsection{Virtual normal bundle and $\mu(c_1(L))$-insertion} 

The (dual) Tanaka-Thomas perfect obstruction theory is given by \cite{TT1}
$$
E_{\mathrm{TT}}^{\mdot \vee} = R\hom_\pi(\E,\E \otimes K_S \otimes \t)_0 - R\hom_\pi(\E,\E)_0.
$$
Using \eqref{Emono}, the class of $E_{\mathrm{TT}}^{\mdot \vee}|_{S_\beta^{[n_0,n_1]}}$ in $K^{0}_{\C^*}(S^{[n_0,n_1]}_{\beta})$ equals the restriction of the following element of $K^{0}_{\C^*}(S^{[n_0]} \times S^{[n_1]}  \times |\beta|)$ 
\begin{align*}
V_{n_0,n_1,\beta}:=\,&R\hom_\pi( \I_0, \I_1(\beta) \otimes \O(1)) +R\Gamma(\O_S) \otimes \O \\
&- R\hom_\pi( \I_0, \I_0) - R\hom_\pi( \I_1, \I_1)  \\
&+R\hom_\pi(\I_1(\beta) \otimes \O(1), \I_0 \otimes K_S^2 \otimes \t^2) -R\Gamma(K_S \otimes \t) \otimes \O \\
&+ R\hom_\pi(\I_0, \I_0 \otimes K_S \otimes \t) + R\hom_\pi(\I_1, \I_1 \otimes K_S \otimes \t) \\
&- R\hom_\pi( \I_0, \I_1(\beta) \otimes K_S^* \otimes \t^{-1}) - R\hom_\pi(  \I_1(\beta) \otimes K_S^* \otimes \t^{-1}, \I_0),
\end{align*}
where lines 1--2 have $\C^*$-weight zero and lines 3--5 have non-zero $\C^*$-weight. We denote by $(\cdot)^{\mov}$ the weight $\neq 0$ part of a complex and by $(\cdot)^{\C^*}$ the weight zero part. Therefore, on  $S_{\beta}^{[n_0,n_1]}$, the virtual normal bundle $N^{\vir}$ and $\C^*$-localized perfect obstruction theory are given by
\begin{align*}
N^{\vir} := (E_{\mathrm{TT}}^{\mdot \vee})^{\mov} &= V_{n_0,n_1,\beta}^{\mov}|_{S_\beta^{[n_0,n_1]}}, \\
(E_{\mathrm{TT}}^{\mdot \vee})^{\C^*} &= V_{n_0,n_1,\beta}^{\C^*}|_{S_{\beta}^{[n_0,n_1]}}.
\end{align*}

Finally, we write
$$
V_{n_0,n_1}:=V_{n_0,n_1,\beta}|_{S^{[n_0]} \times S^{[n_1]} \times \{\pt\}} \in K^{0}_{\C^*}(S^{[n_0]} \times S^{[n_1]}).
$$
This restriction essentially amounts to removing $\O(1)$ from the expression of $V_{n_0,n_1,\beta}$. Using Theorem \ref{GT}, we conclude that the contribution of $S_{\beta}^{[n_0,n_1]}$ to $\chi(N,\widehat{\O}_{N}^{\vir} \otimes \mu(L))$ equals
\begin{align*}
\SW(\beta) \cdot \int_{S^{[n_0]} \times S^{[n_1]}} e\big( R\Gamma(\beta) \otimes \O - R\hom_{\pi}(\I_0,\I_1(\beta)\big) \\ 
\cdot \frac{\ch(\sqrt{ \det (V_{n_0,n_1})^{\vee}})}{\ch(\Lambda_{-1} (V_{n_0,n_1}^{\mov})^{\vee})} \, e^{\mu(c_1(L))} \, \td(V_{n_0,n_1}^{\C^*}).
\end{align*}
 Here we used that $\mu(c_1(L)) =  \pi_*(\pi_S^* c_1(L) \cap (-\ch_2(\EE) + \frac{1}{4}c_1(\EE)^2)$ 
 restricted to $S_{\beta}^{[n_0,n_1]} $ also pulls back from an expression on $S^{[n_0]} \times S^{[n_1]} \times |\beta|$. On 
 $$
 S^{[n_0]} \times S^{[n_1]} \times \{\pt\} \subset S^{[n_0]} \times S^{[n_1]} \times |\beta|
 $$
 this expression is given by
 \begin{align*}
 \mu(c_1(L)) = \,&\pi_*\left(\pi_S^* c_1(L) \cdot (-\ch_2(\I_0) - \ch_2(\I_1)) \cap [ S^{[n_0]} \times S^{[n_1]} \times S]\right) \\
 &- \frac{1}{4} \int_S L \cdot \Big(\frac{c_1-\beta+K_S}{2} \Big)^2 - \frac{1}{4} \int_S L \cdot \Big(\frac{c_1+\beta-K_S}{2} - t \Big)^2 \\
 &+ \frac{1}{2} \int_S L \cdot \Big(\frac{c_1-\beta+K_S}{2} \Big)\Big(\frac{c_1+\beta-K_S}{2} - t \Big) \\
=\, &\pi_* \left(\pi_S^* c_1(L) \cdot (-\ch_2(\I_0) - \ch_2(\I_1)) \cap  [ S^{[n_0]} \times S^{[n_1]} \times S] \right) \\
&+ \frac{t}{2} \int_S L \cdot ( \beta - K_S),
 \end{align*}
where the equivariant integrals $\int_S(\cdots)  \in K^0_{\C^*}(\pt) = \Z[t^{\pm 1}]$ are multiplied with the fundamental class $[S^{[n_0]} \times S^{[n_1]}]$ and we are suppressing some Poincar\'e duals. Exponentiating and using  $y:=e^t$ gives 
$$
e^{\mu(c_1(L))} = y^{\frac{1}{2} L(\beta - K_S)} e^{\pi_*\left(\pi_S^* c_1(L) \cdot (-\ch_2(\I_0) - \ch_2(\I_1)) \cap  [S^{[n_0]} \times S^{[n_1]} \times S]\right)}.
$$

\subsection{Universal series} \label{mon3}

Let $S$ be any smooth projective surface not necessarily satisfying $b_1(S) = 0$ and $p_g(S)>0$. For any $L,\beta \in \Pic(S)$ and $n_0,n_1$, the expressions 
$$
V_{n_0,n_1},  \, \mu(c_1(L)) \in K^0_{\C^*}(S^{[n_0]} \times S^{[n_1]})
$$ 
are defined as in the previous paragraph. We define
\begin{align*}
\sfZ_S^{\mon}(L,\beta,y,q) :=y^{-\frac{1}{2} L(\beta-K_S)} \Bigg( \frac{-1}{y^{\frac{1}{2}}+y^{-\frac{1}{2}}} \Bigg)^{-\chi(\beta-K_S)} (y^{\frac{1}{2}} - y^{-\frac{1}{2}})^{-\chi(\beta) + \chi(\O_S)} \\
\cdot \sum_{n_0,n_1 \geq 0} q^{n_0+n_1} \int_{S^{[n_0]} \times S^{[n_1]}} e\big( R\Gamma(\beta) \otimes \O - R\hom_{\pi}(\I_0,\I_1(\beta)\big) \\ 
\cdot \frac{\ch(\sqrt{ \det (V_{n_0,n_1})^{\vee}})}{\ch(\Lambda_{-1} (V_{n_0,n_1}^{\mov})^{\vee})} \, e^{\mu(c_1(L))} \, \td(V_{n_0,n_1}^{\C^*}).
\end{align*}
Here the first line is a normalization factor which ensures that 
$$
\sfZ_S^{\mon}(L,\beta,y,q) \in 1 + q \, \Q[y^{\pm \frac{1}{2}}] [[q]].
$$
The normalization factor can be computed as follows. Putting $n_0=n_1=0$, the definition of $V_{n_0,n_1}$ gives
\begin{align}
\begin{split} \label{V00}
V_{0,0}=&R\Gamma(\O_S(\beta)) - R\Gamma(\O_S)  +R\Gamma(\O_S(-\beta+ 2K_S) \otimes \t^2) \\
&+R\Gamma(\O_S(K_S) \otimes \t) - R\Gamma(\O_S(\beta-K_S) \otimes \t^{-1}) - R\Gamma(\O_S(-\beta+K_S) \otimes \t).
\end{split}
\end{align}
Using
$$
\frac{\ch(\sqrt{L^*})}{\ch(\Lambda_{-1} L^*)} = \frac{1}{e^{\frac{1}{2} c_1(L)} - e^{-\frac{1}{2} c_1(L)}}
$$
combined with Serre duality and $y = e^t$, we obtain
\begin{align*}
&e^{\mu(c_1(L))} \, \frac{\ch(\sqrt{ \det (V_{0,0})^{\vee}})}{\ch(\Lambda_{-1} (V_{0,0}^{\mov})^{\vee})} \, \td(V_{0,0}^{\C^*})  \\
&=y^{\frac{1}{2} L(\beta-K_S)} \Bigg( \frac{y^{-\frac{1}{2}} - y^{\frac{1}{2}}}{y-y^{-1}} \Bigg)^{\chi(\beta-K_S)} (y^{\frac{1}{2}} - y^{-\frac{1}{2}})^{\chi(\beta) - \chi(\O_S)}.
\end{align*}

The generating series $\sfZ_S^{\mon}(L,\beta,y,q)$ has the following universal property.
\begin{lemma} \label{B}
There exist universal functions 
$$
B_1(y,q),\ldots , B_{7}(y,q) \in 1+q \, \Q[y^{\pm \frac{1}{2}}][[q]]
$$ 
such that for any smooth projective surface $S$ and $L,\beta \in A^1(S)$ we have
$$
\sfZ_S^{\mon}(L,\beta,y,q) = B_1^{L^2} B_2^{L \beta} B_3^{\beta^2} B_4^{LK_S} B_5^{\beta K_S} B_6^{K_S^2} B_7^{\chi(\O_S)}.
$$
\end{lemma}
\begin{proof}
The case $L=\O_S$ is proved (for any rank $r$) in \cite[Sect.~8]{Laa2}. The strategy is similar to the proof of Proposition \ref{A}: \\

\noindent \textbf{Step 1: Multiplicativity.} Let $S = S' \sqcup S''$, where $S', S''$ are possibly disconnected smooth projective surfaces. Let $L,\beta \in A^1(S)$ and define $L':=L|_{S'}$, $\beta':=\beta|_{S'}$, $L'':=L|_{S''}$, and $\beta'':=\beta|_{S''}$. Then
\begin{equation*} 
\sfZ_S^{\mon}(L,\beta,y,q) = \sfZ_{S'}^{\mon}(L',\beta',y,q) \, \sfZ_{S''}^{\mon}(L'',\beta'',y,q).
\end{equation*}
The only new feature compared to \cite[Sect.~8]{Laa2} is the insertion 
$$
\pi_*\left(\pi_S^* c_1(L) \cdot (-\ch_2(\I_0) - \ch_2(\I_1)) \cap  [S^{[n_0]} \times S^{[n_1]}]\right),
$$ 
which we discussed in Lemma \ref{mult}. \\

\noindent \textbf{Step 2: Universality.} This is proved as in Lemma \ref{A}.
\end{proof}

\begin{lemma} \label{C}
Let $S$ be a smooth projective surface with $b_1(S) = 0$, $p_g(S) >0$, and $L \in \Pic(S)$. Let $H, c_1,c_2$ be chosen such that there exist no rank 2 strictly Gieseker $H$-semistable Higgs sheaves on $S$ with Chern classes $c_1,c_2$. For $\vd$ given by \eqref{vd}, the monopole contribution to $\chi(N, \widehat{\O}^{\vir}_N \times \mu(L))$ is given by the coefficient of $(-x)^{\vd}$ of 
\begin{align*} 
 &\sum_{\beta\in H^2(S,\Z)} \, \delta_{c_1,K_S-\beta} \, \SW(\beta) \, B_1(y,x^4)^{L^2}  \Bigg( y^{\frac{1}{2}} B_2(y,x^4) \Bigg)^{L \beta} \\
 &\cdot \Bigg(  \Bigg(\frac{-1}{y^{\frac{1}{2}}+y^{-\frac{1}{2}}}\Bigg)^{\frac{1}{2}} (y^{\frac{1}{2}} - y^{-\frac{1}{2}})^{\frac{1}{2}} (-x)^{-1} B_3(y,x^4) \Bigg)^{\beta^2} \\
 &\cdot \Bigg( y^{-\frac{1}{2}} B_4(y,x^4) \Bigg)^{LK_S} \Bigg( \Bigg(\frac{-1}{y^{\frac{1}{2}}+y^{-\frac{1}{2}}} \Bigg)^{-\frac{3}{2}} (y^{\frac{1}{2}} - y^{-\frac{1}{2}})^{-\frac{1}{2}} (-x)^2 B_5(y,x^4) \Bigg)^{\beta K_S} \\
 &\cdot \Bigg( \Bigg(\frac{-1}{y^{\frac{1}{2}}+y^{-\frac{1}{2}}}\Bigg) (-x)^{-1} B_6(y,x^4) \Bigg)^{K_S^2} \Bigg( \Bigg( \frac{-1}{y^{\frac{1}{2}}+y^{-\frac{1}{2}}} \Bigg) (-x)^{-3} B_7(y,x^4) \Bigg)^{\chi(\O_S)}.
\end{align*}
\end{lemma}
\begin{proof}
By Remark \ref{ignorestab},  we sum the contributions to the invariant of $S_\beta^{[n_0,n_1]}$ for all $\beta \in H^2(S,\Z)$, $n_0, n_1 \in \Z_{\geq 0}$ such that $c_1 + \beta - K_S \in 2H^2(S,\Z)$ and 
$$
c_2 =  n_0+n_1+\Big( \frac{c_1-\beta+K_S}{2} \Big) \Big( \frac{c_1+\beta-K_S}{2} \Big),
$$
or, equivalently, $\vd = 4(n_0+n_1) - (\beta-K_S)^2 - 3 \chi(\O_S)$. As shown in \cite[Sect.~8]{Laa2}, this gives $\sum_{\beta\in H^2(S,\Z)} \, \delta_{c_1,K_S-\beta} \, \SW(\beta) \cdot  (\cdots)$, where $\delta_{a,b}$ was defined in \eqref{defdelta}, and $(\cdots)$ equals the coefficient of $(-1)^{\vd} x^{\vd}$ of
$$
(-1)^{\vd} y^{\frac{1}{2} L(\beta - K_S)} \Bigg( \frac{y^{-\frac{1}{2}} - y^{\frac{1}{2}}}{y-y^{-1}} \Bigg)^{\chi(\beta-K_S)} (y^{\frac{1}{2}} - y^{-\frac{1}{2}})^{\chi(\beta) - \chi(\O_S)}  \sfZ_S^{\mon}(L,\beta,y,x^4).
$$ 
Lemma \ref{B} then gives $\sfZ_S^{\mon}(L,\beta,y,x^4)$ in terms of the universal series $B_i$.
\end{proof}

\begin{proof}[Proof of Theorem \ref{thm1}]
There are finitely many $\beta\in H^2(S,\Z)$ for which $\SW(\beta) \neq 0$. These classes satisfy $\beta^2 = \beta K_S$ \cite[Prop.~6.3.1]{Moc}. The theorem follows by defining $C_1:=B_7$, $C_2:=B_6$, $C_3:=B_1$, $C_4:=B_4$, $C_5 := B_3 B_5$, $C_6:=B_2$.
\end{proof}

\subsection{Reduction to toric surfaces} 

Consider the following 7 choices of $(S,L,\beta)$ for which the corresponding vectors of Chern numbers $(L^2,\ldots,\chi(\O_S))$ are $\Q$-independent:
\begin{align*}
(S,L,\beta) = &(\PP^2,\O,\O), \\ 
&(\PP^2,\O(-3),\O), \\ 
&(\PP^2,\O(-6),\O), \\ 
&(\PP^2,\O,\O(6)), \\ 
&(\PP^2,\O,\O(-6)), \\ 
&(\PP^2,\O(-3),\O(-6)), \\ 
&(\PP^1 \times \PP^1,\O,\O).
\end{align*}
In each case, localization (as in Section \ref{toricA}) reduces the series $\sfZ_S^{\mon}(L,\beta,y,q)$ to a purely combinatorial expression. In this way, we determined the universal series $B_1, \ldots, B_7$ modulo $q^{15}$. For our calculations, we used (and slightly adapted) a SAGE program of Laarakker, which was used for the calculation of $K$-theoretic Vafa-Witten invariants in \cite{Laa2}. Using the definitions of $C_1, \ldots, C_6$ in terms of $B_1, \ldots, B_7$, we obtain \eqref{Cpart1} and \eqref{Cpart2}.

 \subsection{K3 surfaces} \label{mon4}
 
In this section, we consider $\sfZ_S^{\mon}(L,\beta,y,q)$ when $S$ is a K3 surface and $\beta = 0$. Note that $0$ is the only Seiberg-Witten basic class of a K3 surface and $\SW(0) = 1$. Let $\iota : S_0^{[n_0,n_1]} \hookrightarrow S^{[n_0]} \times S^{[n_1]}$ be the natural inclusion. Laarakker \cite[Sect.~10]{Laa2} observes that 
 \begin{equation} \label{diagcycle}
 \iota_* [S_0^{[n_0,n_1]}]^{\vir} = \left\{ \begin{array}{cc} \Delta_* S^{[n]} & \textrm{when \ } n_0=n_1=n \\ 0 & \textrm{otherwise,} \end{array} \right.
 \end{equation}
 where $\Delta : S^{[n]} \hookrightarrow S^{[n]} \times S^{[n]}$ is the diagonal embedding. In other words, only universally thickened nestings $Z_0 = Z_1$ contribute to the invariants.\footnote{The case $n_0=n_1=n$ appears in \cite{GSY1}.} This fact is explained geometrically using cosection localization in \cite[Sect.~5.3]{Tho}. This gives a simplication of $V_{n,n,0}$ (derived in \cite[Sect.~10]{Laa2} for any rank $r$)
 \begin{equation} \label{diagV}
 \Delta^* V_{n,n} = T_{S^{[n]}} + T_{S^{[n]}} \otimes \t^{-1} - T_{S^{[n]}} \otimes \t - T_{S^{[n]}} \otimes \t^2 + V_{0,0},
 \end{equation}
 where $V_{0,0}$ is the normalization term \eqref{V00}, which should be viewed as pulled back from $S^{[n]} \rightarrow \pt$. Using \eqref{diagcycle} and \eqref{diagV}, Laarakker expresses the universal function $C_1$ of Theorem \ref{thm1} in terms of
$$
\chi_{y}(S^{[n]}) = \chi(S^{[n]},\Lambda_{y} \Omega_{S^{[n]}}), \quad \textrm{where \ } S=\textrm{K3}.
$$
In turn, $\chi_y$-genera of Hilbert schemes of points on K3 surfaces were calculated by the first named author and W.~Soergel \cite{GS}.

Recently, using Borisov-Libgober's proof of the Dijkgraaf-Moore-Verlinde-Verlinde formula \cite{BL}, the first named author found a formula for elliptic genera, with values in a line bundle, of Hilbert schemes of points on surfaces \cite{Got2}. We briefly discuss this result. Let $S$ be any smooth projective surface (not necessarily K3) and $L \in \Pic(S)$. The determinant line bundle on $S^{[n]}$ is $\mu(L) := \det( (L-\O_S)^{[n]} )$.
Its first Chern class is described as follows. Consider projections from the universal subscheme $\cZ \subset S^{[n]} \times S$
\begin{displaymath}
\xymatrix
{
& \cZ \ar[dl]_p \ar[dr]^q & \\
S^{[n]} & & S.
}
\end{displaymath}
Then
\begin{equation} \label{muHilb}
c_1(\mu(L)) = \mu(c_1(L)) := p_* q^* c_1(L) \in H^2(S^{[n]},\Z).
\end{equation}
Specialized to $\chi_y$-genera the results of \cite{Got2} imply: 
\begin{theorem}[G\"ottsche] \label{twistedchiy}
Let $S$ be a smooth projective surface and $L \in \Pic(S)$. Then
\begin{align*}
&\sum_{n=0}^{\infty} \chi(S^{[n]},\Lambda_{-y} \Omega_{S^{[n]}} \otimes  \mu(L)) \, (q y^{-1})^n = \\
&\left( \prod_{n=1}^{\infty} \frac{1}{(1-q^{n})^{10}(1-q^{n}y)(1-q^{n}y^{-1})} \right)^{\chi(\O_S)} \left( \prod_{n=1}^{\infty} (1-q^n) \right)^{K_S^2}  \\
&\left(\prod_{n=1}^{\infty} \left(\frac{(1-q^{n})^2}{(1-q^{n}y)(1-q^{n}y^{-1})}\right)^{n^2}\right)^{\frac{L^2}{2}} \prod_{n= 1}^\infty\left( \left(\frac{1-q^ny^{-1}}{1-q^ny}\right)^n\right)^{\frac{LK_S}{2}}.
\end{align*}
\end{theorem}
Just like Laarakker requires G\"ottsche-Soergel's result to determine the monopole contribution to $\chi(N,\widehat{\O}_{N}^{\vir})$ for a K3 surface, we will require Theorem \ref{twistedchiy} to determine the monopole contribution to $\chi(N,\widehat{\O}_{N}^{\vir} \otimes \mu(L) )$ for a K3 surface.

Adapting an argument from \cite{GNY2} and combining with Theorem \ref{twistedchiy}, Conjecture \ref{conj2} (and hence Conjecture \ref{conj1}) are proved for K3 surfaces in \cite{Got2}. We now use \eqref{diagcycle}, \eqref{diagV}, and Theorem \ref{twistedchiy} to prove Theorem \ref{thm2}. 
\begin{proof}[Proof Theorem \ref{thm2}]
Let $S$ be a K3 surface. The case $L = \O_S$ was done in \cite{Laa2} and gives $C_1$. Let $L \in \Pic(S)$ be arbitrary. It is useful to work with $V_{n,n}^{\circ} := V_{n,n} - V_{0,0}$, where $V_{0,0}$ is the normalization factor \eqref{V00} pulled back along $S^{[n]} \times S^{[n]} \rightarrow \pt$. For $S$ a K3 surface and $\beta=0$, \eqref{diagcycle} and \eqref{diagV} imply
\begin{align*}
&\sfZ_S^{\mon}(L,0,y,q) =  \sum_{n} q^{2n} \int_{S^{[n]}} e^{\mu(2c_1(L))} \, \Delta^* \frac{\ch(\sqrt{ \det (V_{n,n}^{\circ})^{\vee}})}{\ch(\Lambda_{-1} (V_{n,n}^{\circ})^{\mov \vee})} \, \td((V_{n,n}^{\circ})^{\C^*}),
\end{align*}
where we used 
\begin{align*}
\Delta^* \pi_* \left( \pi_S^* c_1(L) \cdot (-\ch_2(\I_0) -\ch_2(\I_1)) \cap [S^{[n]} \times S^{[n]} \times S] \right) &= \pi_*( \pi_S^* c_1(L) \cap (2[\cZ])) \\
&= \mu(2c_1(L)), 
\end{align*}
where $\cZ \subset S^{[n]} \times S$ is the universal subscheme and $\mu(c_1(L))$ is defined by \eqref{muHilb}.

We require two identities from \cite{Tho}. By \cite[Prop.~2.6]{Tho}, the canonical square root is given by
$$
\sqrt{ \det (V_{n,n}^{\circ})^{\vee}} = \left( \det (V_{n,n}^{\circ})^{\vee} \right)^{\geq 0} \cdot \t^{\frac{1}{2} r_{\geq 0}},
$$
where $(\cdot)^{\geq 0}$ denotes the part with non-negative $\C^*$-weight and $r_{\geq 0}$ is its rank. Moreover, for any complex $E$ we have \cite[(2.28)]{Tho}
\begin{equation} \label{2.28}
\Lambda_{-1} E^\vee \cong (-1)^{\rk E} \Lambda_{-1} E \otimes \det E^\vee.
\end{equation}

Pulling back along $\Delta : S^{[n]} \hookrightarrow S^{[n]} \times S^{[n]}$ and using \eqref{diagV} yields
$$
\Delta^* \sqrt{ (\det V_{n,n}^{\circ})^{\vee}} = \det \left( \Omega_{S^{[n]}} + \Omega_{S^{[n]}} \otimes \t \right) \cdot \t^{2n} = \det(\Omega_{S^{[n]}}) \cdot \det(\Omega_{S^{[n]}} \otimes \t)  \cdot \t^{2n}.
$$
Furthermore
$$
\Delta^* \frac{1}{\Lambda_{-1} (V_{n,n}^{\circ})^{\mov \vee}} = \frac{\Lambda_{-1} (\Omega_{S^{[n]}} \otimes \t^{-1})}{\Lambda_{-1} (\Omega_{S^{[n]}} \otimes \t)} \cdot \Lambda_{-1} (\Omega_{S^{[n]}} \otimes \t^{-2}).
$$
Hence 
\begin{align*}
\Delta^* \frac{\sqrt{ \det (V_{n,n}^{\circ})^{\vee}}}{\Lambda_{-1} (V_{n,n}^{\circ})^{\mov \vee}} =&\,  \det(\Omega_{S^{[n]}}) \cdot \frac{\det(\Omega_{S^{[n]}} \otimes \t)}{\Lambda_{-1} (\Omega_{S^{[n]}} \otimes \t)}  \cdot \t^{2n} \cdot \Lambda_{-1} (\Omega_{S^{[n]}} \otimes \t^{-1}) \cdot  \Lambda_{-1} (\Omega_{S^{[n]}} \otimes \t^{-2}) \\
=&\,  \det(\Omega_{S^{[n]}})  \cdot \t^{2n} \cdot \frac{\Lambda_{-1} (\Omega_{S^{[n]}} \otimes \t^{-1})}{{\Lambda_{-1} (T_{S^{[n]}} \otimes \t^{-1})}}  \cdot \Lambda_{-1}(\Omega_{S^{[n]}} \otimes \t^{-2}) \\
=&\,  \t^{2n} \cdot \Lambda_{-1}(\Omega_{S^{[n]}} \otimes \t^{-2}),
\end{align*}
where the second equality uses \eqref{2.28}, the third equality uses $T_{S^{[n]}} \cong \Omega_{S^{[n]}}$ (because $S^{[n]}$ is holomorphic symplectic), and the last equation uses $K_{S^{[n]}} \cong \O$. Using $y:=e^t$ and Serre duality (see also \cite[Rem.~4.13]{FG}), we find 
\begin{align*}
\sfZ_S^{\mon}(L,0,y,q) &= \sum_{n=0}^{\infty} y^{2n} \chi(S^{[n]}, \Lambda_{-1}(\Omega_{S^{[n]}} \otimes \t^{-2}) \otimes \mu(L \otimes L)) \, q^{2n} \\
&= \sum_{n=0}^{\infty} y^{2n} \chi(S^{[n]}, \Lambda_{-y^{-2}} \Omega_{S^{[n]}} \otimes \mu(L \otimes L)) \, q^{2n} \\
&= \sum_{n=0}^{\infty} y^{-2n} \chi(S^{[n]},  \Lambda_{-y^{2}} \Omega_{S^{[n]}} \otimes \mu(L^* \otimes L^*)) \, q^{2n}
\end{align*}
The result follows from Theorem \ref{twistedchiy} and Lemmas \ref{B}, \ref{C}.
\end{proof}

\subsection{Higher rank} \label{sec:higher}

The methods of Section \ref{mon1}--\ref{mon4} generalize to any rank $r$. Let $S$ be any smooth projective surface with $b_1(S) = 0$ and $p_g(S)>0$. Let $N:=N_S^H(r,c_1,c_2)$. Suppose there are no rank $r$ strictly Gieseker $H$-semistable Higgs sheaves on $S$ with Chern classes $c_1,c_2$. Consider the components of $N$ containing Higgs sheaves $(E,\phi)$ such that
\begin{equation*} 
E = E_0 \oplus E_1 \otimes \t^{-1} \oplus \cdots \oplus E_{r-1} \otimes \t^{-(r-1)}
\end{equation*}
and $\rk E_0 = \cdots = \rk E_{r-1}=1$. We denote the union of such components by $M_{1^r}$. These components are described by Gholampour-Thomas in terms of nested Hilbert schemes \cite{GT1, GT2} (see also \cite{Laa2})
$$
S_{\beta_1, \ldots, \beta_{r-1}}^{[n_0, \ldots, n_r]} \subset S^{[n_0]} \times \cdots \times S^{[n_{r-1}]} \times |\beta_1| \times \cdots \times |\beta_{r-1}|.
$$
Let $L \in \Pic(S)$ and replace $c_2(\EE) - \frac{1}{4} c_1(\EE)^2$ by $c_2(\EE) - \frac{r-1}{2r} c_1(\EE)^2$ in definitions \eqref{muinsert}, \eqref{muinsertmon}. Let 
$$
\vd:=2rc_2 - (r-1)c_1^2 - (r^2-1)\chi(\O_S).
$$ 
Then the contribution of $M_{1^r}$ to $\chi(N,\widehat{\O}_N^{\vir} \otimes \mu(L))$ is given by the coefficient of $(-x)^{\vd}$ of
\begin{align*}
&\widetilde{C}_1^{(r)}(y,x^{2r})^{\chi(\O_S)} \, \widetilde{C}_2^{(r)}(y,x^{2r})^{K_S^2} \, \widetilde{C}_3^{(r)}(y,x^{2r})^{L^2} \, \widetilde{C}_4^{(r)}(y,x^{2r})^{LK_S} \\
&\cdot \sum_{(a_1, \ldots, a_{r-1}) \in H^2(S,\Z)^{r-1}} \delta_{c_1,K_S-a_1, \ldots,K_S - a_{r-1}} \prod_{i=1}^{r-1} \SW(a_i)  \, \widetilde{C}_{5i}^{(r)}(y,x^{2r})^{a_i K_S} \, \widetilde{C}_{6i}^{(r)}(y,x^{2r})^{a_i L} \\
&\cdot \prod_{i<j} \widetilde{C}_{7ij}^{(r)}(y,x^{2r})^{a_i a_j},
\end{align*}
where $\widetilde{C}_i^{(r)}, \, \widetilde{C}_{ij}^{(r)}, \, \widetilde{C}_{ijk}^{(r)}$ are universal series in $\Q(y^{\frac{1}{2}})(\!(x)\!)$ and
$$
\delta_{a,b_1, \ldots, b_{r-1}} :=  \# \left\{ \gamma \in H^2(S,\Z) \,: \, a-\sum_{i=1}^{r-1} i b_i =r\gamma \right\}.
$$
We did \emph{not} normalize the universal series to start with 1. Since \cite{Laa2} works in any rank, Section \ref{mon1}--\ref{mon3} readily generalize to the above statement.

Equations \eqref{diagcycle} and \eqref{diagV} have analogs in any rank \cite[Sect.~10]{Laa2}. Define
\begin{align*}
\widetilde{C}_1^{(r)}(y,q) &=  x^{-(r^2-1)} (y^{-\frac{r-1}{2}} + y^{-\frac{r-2}{2}} + \cdots + y^{\frac{r-1}{2}} )^{-1} C_1^{(r)}(y,q), \\
\widetilde{C}_3^{(r)}(y,q) &= C_3^{(r)}(y,q) ,
\end{align*}
then $C_3^{(r)}, C_5^{(r)} \in 1+q\,\Q(y^{\frac{1}{2}})[[q]]$. Generalizing Section \ref{mon4} accordingly yields
\begin{align*}
C_1^{(r)}(y,q) &=  \prod_{n=1}^{\infty} \frac{1}{(1-q^{rn})^{10}(1-q^{rn}y^r)(1-q^{rn}y^{-r})}, \\
C_3^{(r)}(y,q) &= \prod_{n=1}^{\infty} \left(\frac{(1-q^{rn})^2}{(1-q^{rn}y^r)(1-q^{rn}y^{-r})}\right)^{\frac{r^2n^2}{2}},
\end{align*}
where $C_1^{(r)}$ was previously derived in \cite{Laa2, Tho} and $C_3^{(r)}$ is new.

Let $M := M_S^H(r,c_1,c_2)$ and assume there are no rank $r$ strictly Gieseker $H$-semistable sheaves on $S$ with Chern classes $c_1,c_2$. The instanton contribution to $(-1)^{\vd} \chi(N,\widehat{\O}_N^{\vir} \otimes \mu(L))$, which equals $y^{-\frac{\vd}{2}}\chi^{\vir}_{-y}(M,\mu(L))$, is determined in \cite{Got2} for $S$ a K3 surface. It is derived by combining Theorem \ref{twistedchiy} with an adaptation of an argument of \cite{GNY2}. The result is the coefficient of $q^{\vd/2}$ of
$$
\left(\prod_{n=1}^{\infty} \frac{1}{(1-q^{n})^{20}(1-q^{n}y)^2(1-q^{n}y^{-1})^2} \right) \left(\prod_{n=1}^{\infty} \left(\frac{(1-q^{n})^2}{(1-q^{n}y)(1-q^{n}y^{-1})}\right)^{n^2} \right)^{\frac{L^2}{2}}.
$$
Unlike the monopole contribution, these universal series are independent of $r$.

\section{Applications} \label{appl}

In this section, we discuss special cases of Conjectures \ref{conj1} and \ref{conj2}: (1) minimal surfaces of general type, (2) surfaces with disconnected canonical divisor, (3) a blow-up formula, and (4) Vafa-Witten invariants with $\mu$-classes. We denote the formula of Conjecture \ref{conj1}, after some slight rewriting, by
\begin{align} 
\begin{split} \label{defpsi}
&\psi_{S,L,c_1}(x) :=\\
& \frac{2^{2-\chi(\O_S)+K_S^2}}{(1-x^2)^{\chi(L)}} \sum_{a\in H^2(S,\Z)} \SW(a) \, (-1)^{ac_1}(1+x)^{(K_S-a)(L-K_S)}(1-x)^{a(L-K_S)}.
\end{split}
\end{align}

\subsection{Minimal surfaces of general type}

\begin{proposition} \label{prop1}
Let $S$ be a smooth projective surface satisfying $p_g(S)>0$, $b_1(S) = 0$, $K_S \neq 0$, and such that its only Seiberg-Witten basic classes are $0$ and $K_S$. Let $L \in \Pic(S)$ and let $H,c_1,c_2$ be chosen such that there are no rank 2 strictly Gieseker $H$-semistable sheaves on $S$ with Chern classes $c_1, c_2$. Suppose Conjecture \ref{conj1} holds in this setting. Then $\chi^{\vir}(M_S^H(2,c_1,c_2), \mu(L))$ is given by the coefficient of $x^{\vd}$ of 
$$
2^{3 - \chi(\O_S) +K_S^2} \frac{(1+x)^{K_S(L-K_S)}}{(1-x^2)^{\chi(L)}}.
$$
\end{proposition}
\begin{proof}
Since $\SW(0)=1$, we have $\SW(K_S) = (-1)^{\chi(\O_S)}$ \cite[Prop.~6.3.4]{Moc}. By Conjecture \ref{conj1}, $\chi^{\vir}(M_S^H(2,c_1,c_2), \mu(L))$ is given by the coefficient of $x^{\vd}$ of \eqref{defpsi}, which simplifies to
$$
\frac{2^{2 - \chi(\O_S) +K_S^2}}{(1-x^2)^{\chi(L)}} \left[ (1+x)^{K_S(L-K_S)} + (-1)^{c_1K_S+\chi(\O_S)} (1-x)^{K_S(L-K_S)}  \right].
$$
Varying over $c_2$, we put the coefficients of all terms $x^{\vd}$ of $\psi_{S,L,c_1}(x)$ into a generating series as follows. Suppose $\psi_{S,L,c_1}(x) = \sum_{n=0}^{\infty} \psi_n x^n$ and $i = \sqrt{-1}$. Then for $\vd$ given by \eqref{vd}, we have
\begin{align*}
&\sum_{c_2} \mathrm{Coeff}_{x^{\vd}}(\psi_{S,L,c_1}(x)) \, x^{\vd} = \sum_{n \equiv -c_1^2 - 3\chi(\O_S) \mod 4} \psi_n \, x^n \\
&= \sum_{k=0}^{3} \frac{1}{4} i^{k(c_1^2+3\chi(\O_S))} \psi(i^k x) \\
&= 2^{1 - \chi(\O_S) +K_S^2} \left[ \frac{(1+x)^{K_S(L-K_S)}}{(1-x^2)^{\chi(L)}} + (-1)^{c_1^2 + 3\chi(\O_S)} \frac{(1-x)^{K_S(L-K_S)}}{(1-x^2)^{\chi(L)}} \right. \\
&\qquad \left. + i^{c_1^2 + 3\chi(\O_S)} \frac{(1+ix)^{K_S(L-K_S)}}{(1+x^2)^{\chi(L)}} + (-i)^{c_1^2 + 3\chi(\O_S)} \frac{(1-ix)^{K_S(L-K_S)}}{(1+x^2)^{\chi(L)}}  \right],
\end{align*}
where the third equality uses $c_1 K_S \equiv c_1^2 \mod 2$. Now define
$$
\phi_{S,L,c_1}(x) := 2^{3 - \chi(\O_S) +K_S^2} \frac{(1+x)^{K_S(L-K_S)}}{(1-x^2)^{\chi(L)}}.
$$
Then
 \begin{align*}
\sum_{c_2} \mathrm{Coeff}_{x^{\vd}}(\phi_{S,L,c_1}(x)) \, x^{\vd} &= \sum_{n \equiv -c_1^2 - 3\chi(\O_S) \mod 4} \phi_n \, x^n \\
&= \sum_{k=0}^{3} \frac{1}{4} i^{k(c_1^2+3\chi(\O_S))} \phi(i^k x)
\end{align*}
is given by the same expression as above, which proves the proposition.
\end{proof}

\begin{remark}
Examples of surfaces satisfying the conditions of Proposition \ref{prop1} are (1) minimal surfaces of general type satisfying $p_g(S)>0$ and $b_1(S) = 0$ \cite[Thm.~7.4.1]{Mor}, (2) smooth projective surfaces with $b_1(S) = 0$ and containing an irreducible reduced curve $C \in |K_S|$ (e.g.~discussed in \cite[Sect.~6.3]{GK1}).
\end{remark}

\begin{remark}
In general, the formula of Proposition \ref{prop1} only has integer coefficients when $\chi(\O_S) - 3 \leq K_S^2$. For minimal surfaces of general type, this inequality is implied by Noether's inequality $\chi(\O_S) - 3 \leq \frac{1}{2} K_S^2$.
\end{remark}

\begin{corollary}
Let $S$ be a smooth projective surface with $b_1(S) = 0$ and containing a smooth connected curve $C \in |K_S|$ of genus $g$. Let $L \in \Pic(S)$ and let $H,c_1,c_2$ be chosen such that there are no rank 2 strictly Gieseker $H$-semistable sheaves on $S$ with Chern classes $c_1, c_2$. Suppose Conjecture \ref{conj1} holds in this setting. Then $\chi^{\vir}(M_S^H(2,c_1,c_2), \mu(L))$ is given by the coefficient of $x^{\vd}$ of 
$$
2^{3 - \chi(\O_C) - \chi(\O_S)} \frac{(1+x)^{\chi(L|_C)}}{(1-x^2)^{\chi(L)}}.
$$
\end{corollary}

\begin{proof}
We have $g = K_S^2+1$ and $\chi(L|_C) = 1 - g + \deg L|_C$ by Riemann-Roch.
\end{proof}

\subsection{Disconnected canonical divisor}

\begin{proposition}
Let $S$ be a smooth projective surface with $b_1(S) = 0$ and suppose there exists $0 \neq C_1 + \cdots + C_m \in |K_S|$, where $C_1, \ldots, C_m$ are mutually disjoint irreducible reduced curves. Let $L \in \Pic(S)$ and let $H,c_1,c_2$ be chosen such that there are no rank 2 strictly Gieseker $H$-semistable sheaves on $S$ with Chern classes $c_1, c_2$. Suppose Conjecture \ref{conj1} holds in this setting. Then $\chi^{\vir}(M_S^H(2,c_1,c_2), \mu(L))$ is given by the coefficient of $x^{\vd}$ of 
$$
\frac{2^{2 - \chi(\O_S) + K_S^2}}{(1-x^2)^{\chi(L)}} \prod_{j=1}^{m} \left[ (1+x)^{\chi(L|_{C_i})} + (-1)^{C_i c_1 + h^0(N_{C_i/S})} (1-x)^{\chi(L|_{C_i})} \right],
$$
where $N_{C_i / S}$ denotes the normal bundle of $C_i \subset S$.
\end{proposition}
\begin{proof}
We describe the Seiberg-Witten basic classes and invariants for $S$ in this setting \cite[Lem.~6.14]{GK1}. For any $I \subset M:=\{1, \ldots, m\}$, define $C_I := \sum_{i \in I} C_i$ and we write $I \sim J$ when $C_I$ and $C_J$ are linearly equivalent. Also $C_{\varnothing} := 0$. The Seiberg-Witten basic classes of $S$ are precisely $\{C_I\}_{I \subset M}$ and
$$
\SW(C_I) = \# [I] \prod_{i \in I} (-1)^{h^0(N_{C_i/S})},
$$ 
where $\# [I]$ denotes the number of elements of equivalence class $[I]$. Therefore \eqref{defpsi} becomes
\begin{align*}
&\frac{2^{2 - \chi(\O_S) + K_S^2}}{(1-x^2)^{\chi(L)}}  \Bigg( \sum_{[I]} \# [I] \prod_{i \in I} (-1)^{h^0(N_{C_i/S})} \Bigg) (-1)^{C_I c_1} (1+x)^{C_{M \setminus I} (L-K_S)} (1-x)^{C_{I} (L-K_S)} \\
&= \frac{2^{2 - \chi(\O_S) + K_S^2}}{(1-x^2)^{\chi(L)}}  \sum_{I \subset M} \Bigg( \prod_{i \in I} (-1)^{C_i c_1 + h^0(N_{C_i/S})} (1-x)^{C_{i} (L-C_i)} \Bigg) \Bigg( \prod_{i \in M \setminus I} (1+x)^{C_{i} (L-C_i)} \Bigg),
\end{align*}
where we used $K_S = C_M$ and the assumption that the curves $C_i$ are mutually disjoint. The result follows from $\chi(L|_{C_i}) = 1 - g(C_i) + \deg L|_{C_i} =  C_i(L - C_i)$ and expanding the product in the statement of the proposition.
\end{proof}

\subsection{Blow-up formula}

\begin{proposition}
Let $S$ be a smooth projective surface, $\pi : \widetilde{S} \rightarrow S$ the blow-up of $S$ in a point, and $E$ the exceptional divisor. Let $L,c_1 \in \Pic(S)$, $\widetilde{c}_1 = \pi^* c_1 - k E$, and $\widetilde{L} = \pi^*L - \ell E$. Then 
$$
\psi_{\widetilde{S},\widetilde{L}, \widetilde{c}_1}(x) = \frac{1}{2} (1-x^2)^{\binom{\ell+1}{2}} \left[ (1+x)^{\ell+1} + (-1)^k (1-x)^{\ell+1} \right] \psi_{S,L, c_1}(x).
$$
\end{proposition}
\begin{proof}
The Seiberg-Witten basic classes of $\widetilde{S}$ are $\pi^* a$ and $\pi^*a + E$ with corresponding Seiberg-Witten invariant $\SW(a)$, where $a$ runs over all Seiberg-Witten basic classes of $S$ \cite[Thm.~7.4.6]{Mor}. Using $\chi(\O_{\widetilde{S}}) = \chi(\O_S)$, $K_{\widetilde{S}} = \pi^*K_S+E$, $E^2 = -1$, $\chi(\widetilde{L}) = \chi(L) - \binom{\ell+1}{2}$, the proposition follows at once from \eqref{defpsi}.
\end{proof}

\subsection{Vafa-Witten formula with $\mu$-classes}

Let $S$ be a smooth projective surface satisfying $b_1(S) = 0$ and $p_g(S)>0$. In an appendix of \cite{GK1}, the first named author and Nakajima gave a conjectural formula for  
\begin{equation} \label{GKapp}
\sum_{k=0}^{\vd} \int_{[M]^{\vir}} e^{\mu(c_1(L))} \lambda^{\vd - k} c_k(T^{\vir}_M),
\end{equation}
where $M := M_{S}^{H}(2,c_1,c_2)$, $\vd$ is given by \eqref{vd}, and we assume ``stable=semistable''. Here $\lambda$ is a formal parameter. Setting $\lambda=0$ in \eqref{GKapp} gives $e^{\vir}(M)$. Replacing $\lambda$ by $\lambda^{-1}$, then multiplying by $\lambda^{\vd}$, and finally setting $\lambda=0$ gives Donaldson invariants $\int_{[M]^{\vir}} e^{\mu(c_1(L))}$. Therefore \eqref{GKapp} interpolates between Donaldson invariants and virtual Euler characteristics. Let $G_2(q)$ be the Eisenstein series of weight 2 and define
$$
\overline{G}_2(q) = G_2(q) + \frac{1}{24} = \sum_{d=1}^{\infty} \sigma_1(d) \, q^d,
$$
where $\sigma_1(d) = \sum_{d | n} d$. Furthermore, let $\theta_3(q) := \theta_3(q,1)$ and $D := q \frac{d}{d q}$. 

\begin{conjecture}[G\"ottsche-Nakajima] \label{conjGN}
Let $S$ be a smooth projective surface with $p_g(S)>0$, $b_1(S)=0$, and let $L \in \Pic(S)$. Let $H,c_1,c_2$ be chosen such that there are no rank 2 strictly Gieseker $H$-semistable sheaves on $S$ with Chern classes $c_1, c_2$. Let $M:= M_S^H(2,c_1,c_2)$, then 
$$
\sum_{k=0}^{\vd} \int_{[M]^{\vir}} e^{\mu(c_1(L))} \lambda^{\vd - k} c_k(T^{\vir}_M)
$$
is given by the coefficient of $x^{\vd}$ of 
\begin{align*}
&4\left( \frac{1}{2 \overline{\eta}(x^2)^{12}} \right)^{\chi(\O_S)} \left(\frac{ 2 \overline{\eta}(x^4)^2}{\theta_3(x)}\right)^{K_S^2} \left( e^{DG_2(x^2)} \right)^{\frac{(\lambda L)^2}{2}} \left( e^{-2 \overline{G}_2(x^2)} \right)^{\lambda LK_S}\\
&\cdot\sum_{a\in H^2(S,\Z)} (-1)^{c_1 a} \, \SW(a) \,   \left(\frac{\theta_3(x,y^{\frac{1}{2}})}{\theta_3(-x,y^{\frac{1}{2}})}\right)^{aK_S} \left( e^{G_2(x) - G_2(-x)}  \right)^{\frac{\lambda L(K_S-2a)}{2}}.
\end{align*}
\end{conjecture}

Recall that specializing Conjecture \ref{conj2} to  $y=0$ implies Conjecture \ref{conj1} (after replacing $x$ by $xy^{\frac{1}{2}}$, see Section \ref{intro}). We show that specializing Conjecture \ref{conj2} to $y=1$ implies Conjecture \ref{conjGN} (after replacing $x$ by $xy^{\frac{1}{2}}$ and $L$ by $\lambda L (y^{-\frac{1}{2}} - y^{\frac{1}{2}})^{-1}$). In summary: the invariants of this paper interpolate between:
\begin{itemize}
\item Donaldson invariants,
\item virtual Euler numbers of moduli spaces of sheaves,
\item $K$-theoretic Donaldson invariants,
\item $K$-theoretic Vafa-Witten invariants.
\end{itemize}

\begin{proposition} \label{conj2impliesVWmu}
Conjecture \ref{conj2} implies Conjecture \ref{conjGN}.
\end{proposition}

\begin{proof}[Proof of Proposition \ref{conj2impliesVWmu}]
Recall the definition of $y^{-\frac{r}{2}}\mathsf{X}_{-y}(E)$, for any complex $E$ of rank $r$ on $M$, from Section \ref{Doninv}. Suppose $r \geq 0$ and denote by $\{ \cdot \}_r$ the degree $r$ part in $A^*(M)_{\Q}$. Then \cite[Thm.~4.5]{FG} 
$$
\left\{ y^{-\frac{r}{2}}\mathsf{X}_{-y}(E) \right\}_r = c_{r}(E) \mod (1-y),
$$
For $D \in A^1(M)_{\Q}$, we are interested in  
$$
\left\{ \sum_{k=0}^{r} e^{D} \lambda^{r - k} c_k(E) \right\}_r,
$$
which is insertion \eqref{GKapp} for $E = T_{M}^{\vir}$ and $D = \mu(c_1(L))$. We consider 
$$
e^{\lambda D (y^{-\frac{1}{2}} - y^{\frac{1}{2}})^{-1}} y^{-\frac{r}{2}}\mathsf{X}_{-y}(E).
$$
Again using  \cite[Thm.~4.5]{FG}, we find 
$$
\left\{ e^{\lambda D (y^{-\frac{1}{2}} - y^{\frac{1}{2}})^{-1}} y^{-\frac{r}{2}}\mathsf{X}_{-y}(E) \right\}_r = \left\{ \sum_k e^{D} \lambda^{r - k}  c_k(E) y^{-\frac{k}{2}} \right\}_r \mod (1-y).
$$
Hence
\begin{equation} \label{replaceL}
\left\{ e^{\lambda D (y^{-\frac{1}{2}} - y^{\frac{1}{2}})^{-1}} y^{-\frac{r}{2}}\mathsf{X}_{-y}(E) \Big|_{y=1} \right\}_r = \left\{ \sum_{k=0}^{r} e^{D} \lambda^{r - k} c_k(E) \right\}_r.
\end{equation}
Take $E = T_{M}^{\vir}$ and $D = \mu(c_1(L))$. Replacing $L$ by 
$$
\frac{\lambda L}{y^{-\frac{1}{2}} - y^{\frac{1}{2}}}
$$
in Conjecture \ref{conj2} and setting $y=1$ gives the invariants \eqref{GKapp} by equation \eqref{replaceL}. 

This reduces the proof to the following identities
\begin{align*}
D G_2(x^2) &=\lim_{y\to 1} \sum_{n=1}^{\infty}  \frac{n^2}{(y^{-\frac{1}{2}} - y^{\frac{1}{2}})^2} \log \frac{(1-x^{2n})^2}{(1-x^{2n}y) (1-x^{2n} y^{-1})}, \\
\overline G_2(x^2) &= - \frac{1}{2} \lim_{y\to 1} \sum_{n=1}^{\infty} \frac{n}{(y^{-\frac{1}{2}} - y^{\frac{1}{2}})} \log \frac{(1-x^{2n}y^{-1})}{(1-x^{2n}y)}, \\
G_2(x)-G_2(-x) &=\lim_{y\to 1} \sum_{n>0 \atop \textrm{odd}}
\frac{n}{y^{-\frac{1}{2}} - y^{\frac{1}{2}}}  \log \frac{(1-x^n y^{\frac{1}{2}})(1+x^n y^{-\frac{1}{2}})}{(1-x^n y^{-\frac{1}{2}})(1+x^n y^{\frac{1}{2}})}.
\end{align*}
These identities follow from an elementary computation using repeatedly that 
\begin{align*}
\log(1-x)&=-\sum_{n=1}^{\infty} \frac{x^n}{n}, \\
\lim_{y\to 1} \frac{y^{-\frac{n}{2}}-y^{\frac{n}{2}}}{y^{-\frac{1}{2}}-y^{\frac{1}{2}}}&=n.
\end{align*}
Therefore
\begin{align*} 
&\lim_{y\to 1}  \sum_{n=1}^{\infty} \frac{n^2}{(y^{-\frac{1}{2}} - y^{\frac{1}{2}})^2} \log \frac{(1-x^{2n})^2}{(1-x^{2n}y) (1-x^{2n} y^{-1})}=\lim_{y\to 1}\sum_{n,l>0}\frac{n^2 x^{2nl}}{l}\left(\frac{y^{-\frac{l}{2}}-y^{\frac{l}{2}}}{
y^{-\frac{1}{2}}-y^{\frac{1}{2}}}\right)^2 \\
&=\sum_{n,l>0} n^2 l \,  x^{2nl} = DG_2(x^2).
\end{align*}
The other identities follow similarly.
%Similarly
%\begin{align*}
%-\frac{1}{2} \lim_{y\to 1}&  \sum_{n=1}^{\infty}  \frac{n}{(y^{-\frac{1}{2}} - y^{\frac{1}{2}})} \log \frac{(1-x^{2n}y^{-1})}{(1-x^{2n}y)}= \frac{1}{2} \sum_{n,l>0}
% \frac{nx^{2nl}}{l}\frac{y^{-l}-y^{l}}{y^{-\frac{1}{2}}-y^{\frac{1}{2}}}\\&= \frac{1}{2}\sum_{n,l>0}
%\frac{nx^{2nl}}{l} (2l)=\overline G_2(x^2).
%\end{align*}
%Finally
%\begin{align*}
%\lim_{y\to 1}& \sum_{n>0\atop \textrm{odd}} \frac{n}{y^{-\frac{1}{2}} - y^{\frac{1}{2}}}
%\log \frac{(1-x^n y^{\frac{1}{2}})(1+x^n y^{-\frac{1}{2}})}{(1-x^n y^{-\frac{1}{2}})(1+x^n y^{\frac{1}{2}})}\\&=
%\lim_{y\to 1} \sum_{n,l>0\atop n\ \textrm{odd}} \frac{n x^{nl}}{l}\frac{y^{-\frac{l}{2}}-(-1)^l y^{-\frac{l}{2}}-y^{\frac{l}{2}}+(-1)^ly^{\frac{l}{2}}}{y^{-\frac{1}{2}}-y^{\frac{1}{2}}}\\&=
%\lim_{y\to 1}\sum_{n,l>0\atop n,l\ \textrm{odd}} \frac{n x^{nl}}{l}\frac{2y^{-l/2}-2y^{l/2}}{y^{-1/2}-y^{1/2}}=
%\sum_{n,l>0\atop n,l\ \textrm{odd}} \frac{n x^{nl}}{l}2l=G_2(x)-G_2(-x).
%\end{align*}
\end{proof}

{\tt{gottsche@ictp.trieste.it, m.kool1@uu.nl, rwilliam@ictp.it}}

\begin{thebibliography}{AGDP}
%\bibitem[ACGH]{ACGH} E.~Arbarello, M.~Cornalba, P.~A.~Griffiths, and J.~Harris, \textit{Geometry of algebraic curves}, Volume I, Springer-Verlag (1985).
\bibitem[AGDP]{AGDP} J.~E.~Andersen S.~Gukov, and Du Pei, \textit{The Verlinde formula for Higgs bundles}, arXiv:1608.01761.
%\bibitem[AIK]{AIK} A.~Altman, A.~Iarrobino, and S.~Kleiman, \textit{Irreducibility of the compactified Jacobian}, Real and complex singularities (Proc. Ninth Nordic Summer School/NAVF Sympos.~Math., Oslo, 1976) 1--12 (1977).
%\bibitem[Apo]{Apo} T.~M.~Apostol, \textit{Modular functions and Dirichlet series in number theory}, Springer (1967).
%\bibitem[BE]{BE} A.~Braverman and P.~Etingof, \textit{Instanton counting via affine Lie algebras II: from Whittaker vectors to the Seiberg-Witten prepotential}, Studies in Lie theory, Progr.~Math.~243, 61--78, Birkh\"auser, Boston (2006).
\bibitem[Bea]{Bea} A.~Beauville, \textit{Fibr\'es de rang 2 sur les courbes, fibr\'e determinant et fonctions th\^eta}, Bull.~Soc.~Math.~France 116 (1988) 431--448.
\bibitem[BL]{BL} A.~Beauville and Y.~Laszlo, \textit{Conformal blocks and generalized theta functions}, Comm.~Math.~Phys.~164 (1994) 385--419. 
\bibitem[Beh]{Beh} K.~Behrend, \textit{Donaldson-Thomas type invariants via microlocal geometry}, Annals of Math.~170 (2009) 1307--1338.
\bibitem[BS]{BS} A.~Bertram and A.~Szenes, \textit{Hilbert polynomials of moduli spaces of rank 2 vector bundles II}, Topol.~32 (1993) 599--609.
%\bibitem[Beh]{Beh} K.~Behrend, \textit{Gromov-Witten invariants in algebraic geometry}, Invent.~Math.~127 601--617 (1997). arXiv:alg-geom/9601011v1.
%\bibitem[BF]{BF} K.~Behrend and B.~Fantechi, \textit{The intrinsic normal cone}, Invent.~Math.~{\bf{128}} (1997) 45--88. arXiv:9601010
%\bibitem[BGS]{BGS} J.~Brian\c{c}on, M.~Granger, J.-P.~Speder, \textit{Sur le sch\'ema de Hilbert d'une courbe plane}, Ann.~Sci.~de l'~\'Ecole Normale Sup\'erieure 4 14 1--25 (1981).
%\bibitem[Bri]{Bri} T.~Bridgeland, \textit{Fourier-Mukai transforms for elliptic surfaces}, J.~reine angew.~Math.~498 (1998) 115--133.
%\bibitem[BN]{BN} K.~Bringmann and C.~Nazaroglu, \textit{An exact formula for $U(3)$ Vafa-Witten invariants on $\mathbb{P}^2$}, arXiv:1803.09270.
%\bibitem[BL1]{BL1} J.~Bryan, C.~Leung, \textit{The enumerative geometry of K3 surfaces and modular forms}, J.~Amer.~Math.~Soc.~13 371--410 (2000).
%\bibitem[BL]{BL} J.~Bryan and C.~Leung, \textit{Generating functions for the number of curves on abelian surfaces}, Duke Math.~J.~99 311--328 (1999). arXiv:math/9802125v1.
%\bibitem[BLN]{BLN} L.~Baulieu, A.~Losev, and N.~Nekrasov, \textit{Chern-Simons and twisted supersymmetry in various dimensions}, Nuclear Phys.~B 522 (1998) 82--104. 
%\bibitem[Blo]{Blo} S.~Bloch, \textit{Semi-regularity and de Rham cohomology}, Invent.~Math.~17 51--66 (1972).
%\bibitem[BM]{BM} K.~Behrend, Yu.~Manin, \textit{Stacks of stable maps and Gromov--Witten invariants}, Duke Math.~J.~85 1 1--60 (1996).
\bibitem[Bot]{Bot} R.~Bott, \textit{On E.~Verlinde's formula in the context of stable bundles}, Int.~J.~Mod.~Phys.~A6 (1991) 2847--2858.
%\bibitem[BPV]{BPV} W.~Barth, C.~Peters, and A.~van de Ven, \textit{Compact complex surfaces}, Springer-Verlag (1984).
%\bibitem[Bri1]{Bri1} T.~Bridgeland, \textit{An introduction to motivic Hall algebras}, Adv.~Math.~\textbf{229} (2012) 102--138.
%\bibitem[Bri2]{Bri2} T.~Bridgeland, \textit{Hall algebras and curve counting}, JAMS \textbf{24} (2011) 969--998.
%\bibitem[BS]{BS} M.~Beltrametti and A.~J.~Sommese, \textit{Zero cycles and $k$th order embeddings of smooth projective surfaces. With an appendix by Lothar G\"ottsche}, Problems in the theory of surfaces and their classification (Cortona, 1988) Sympos.~Math.~32 33--48 Academic Press (1991).
%\bibitem[CD]{CD} F.~Catanese and O.~Debarre, \textit{Surfaces with $K^2 = 2$, $p_g = 1$, $q = 0$}, J.~Reine Angew.~Math.~395 (1989), 1--55.
%\bibitem[CFK]{CFK} I.~Ciocan-Fontanine and M.~Kapranov, \textit{Virtual fundamental classes via dg-manifolds}, Geom.~Topol.~13 (2009) 1779--1804.
%\bibitem[CK]{CK} H.-l.~Chang and Y.-H.~Kiem, \textit{Poincar\'e invariants are Seiberg-Witten invariants}, Geom.~and Topol.~17 (2013) 1149--1163. 
%\bibitem[DKO]{DKO} M.~D\"urr, A.~Kabanov, and C. Okonek, \textit{Poincar\'e invariants}, Topol.~{\bf{46}} (2007) 225--294. arXiv:0408131
\bibitem[DW]{DW} G.~Daskalopoulos and R.~Wentworth, \textit{Local degeneration of the moduli space of vector bundles and factorization of rank two theta functions. I}, Math.~Ann.~297 (1993) 417--466.
\bibitem[Don]{Don} S.~Donaldson, \textit{Gluing techniques in the cohomology of moduli spaces}, in: Topological methods in modern mathematics (Proc.~of 1991 Stony Brook conference in honour of the sixtieth birthday of J.~Milnor), Publish or Perish.
\bibitem[DN]{DN} J.~M.~Dr\'ezet and M.~S.~Narasimhan, \textit{Groupe de Picard des vari\'et\'es de modules de fibr\'es semi-stables sur les courbes alg\'ebriques}, Invent.~Math.~97 (1989) 53--94.
%\bibitem[Don]{Don} S.~K.~Donaldson, \textit{Irrationality and the $h$-cobordism conjecture}, J.~Diff.~Geom.~26 141--168 (1987).
%\bibitem[DPS]{DPS} R.~Dijkgraaf, J.-.~Park, and B.~J.~Schroers, \textit{$N=4$ supersymemtric Yang-Mills theory on a K\"ahler surface}, hep-th/9801066 ITFA-97-09.
%\bibitem[EG]{EG} D.~Edidin and W.~Graham, \textit{Riemann-Roch for equivariant Chow groups}, Duke Math.~J.~{\bf{102}} (2000) 567--594. arXiv:9905081
%\bibitem[EZ]{EZ} M.~Eichler and D.~Zagier, \textit{The theory of Jacobi forms}, Progress in Math.~55, Birkh\"auser (1985).
\bibitem[EGL]{EGL} G.~Ellingsrud, L.~G\"ottsche, and M.~Lehn, \textit{On the cobordism class of the Hilbert scheme of a surface}, Jour.~Alg.~Geom.~10 (2001) 81--100. %arXiv:math/9904095v1.
%\bibitem[Eis]{Eis} D.~Eisenbud, \textit{Commutative algebra, with a view towards algebraic geometry}, Springer (1999). 
\bibitem[Fal]{Fal} G.~Faltings, \textit{A proof of the Verlinde formula}, J.~Alg.~Geom.~3 (1994) 347--374.
\bibitem[FG]{FG} B.~Fantechi and L.~G\"ottsche, \textit{Riemann-Roch theorems and elliptic genus for virtually smooth schemes}, Geom.~Topol.~14 (2010) 83--115.
%\bibitem[FM]{FM} R.~Friedman and J.~W.~Morgan, \textit{Obstruction bundles, semiregularity and Seiberg-Witten invariants}, Comm.~Anal.~Geom. 7 (1999) 451--495.
%\bibitem[Fri1]{Fri1} R.~Friedman, \textit{Vector bundles and $SO(3)$-invariants for elliptic surfaces III}, arXiv:alg-geom/9308004.
%\bibitem[Fri2]{Fri2} R.~Friedman, \textit{Donaldson and Seiberg-Witten invariants of algebraic surfaces}, in Proc.~Pure Math.~(Alg.~Geom.~Santa Cruz 1995), editors: J.~Koll\'ar, R.~Lazarsfeld, D.~R.~Morrison, AMS (1997).
%\bibitem[FM]{FM} R.~Friedman and J.~W.~Morgan, \textit{Obstruction bundles, semiregularity and Seiberg-Witten invariants}, Comm.~Anal.~Geom.~7 451--495 (1999).
%\bibitem[Ful]{Ful} W.~Fulton, \textit{Intersection theory}, Springer-Verlag (1998).
%\bibitem[GH]{GH} Ph.~Griffiths, J.~Harris, \textit{Principles of algebraic geometry}, Wiley Classics Library (1994).
%\bibitem[GSY]{GSY} A.~Gholampour, A.~Sheshmani, and S.-Y.~Yau, \textit{Localized Donaldson-Thomas theory of surfaces}, arXiv:1701.08902.
%\bibitem[Got]{Got} L.~G\"ottsche, \textit{Theta functions and Hodge numbers of moduli spaces of sheaves on rational surfaces}, Comm.~Math.~Phys.~206 (1999) 105--136.
\bibitem[GSY1]{GSY1} A.~Gholampour, A.~Sheshmani, and S.-T.~Yau, \textit{Nested Hilbert schemes on surfaces: Virtual fundamental class}, Adv.~Math.~365 (2020) 107046.
\bibitem[GSY2]{GSY2} A.~Gholampour, A.~Sheshmani, and S.-T.~Yau, \textit{Localized Donaldson-Thomas theory of surfaces}, Amer.~Jour.~Math.~142 (2020) 405--442.
\bibitem[GT1]{GT1} A.~Gholampour and R.~P.~Thomas, \textit{Degeneracy loci, virtual cycles and nested Hilbert schemes I}, Tunisian Jour.~Math.~2 (2020) 633--665.
\bibitem[GT2]{GT2} A.~Gholampour and R.~P.~Thomas, \textit{Degeneracy loci, virtual cycles and nested Hilbert schemes II}, to appear in Compos.~Math., arXiv:1902.04128.
\bibitem[Got1]{Got1} L.~G\"ottsche, \textit{Verlinde-type formulas for rational surfaces}, JEMS 22 (2020) 151--212.
\bibitem[Got2]{Got2} L.~G\"ottsche, \textit{Refined Verlinde formulas for Hilbert schemes of points and moduli of sheaves on K3 surfaces}, arXiv:1903.03874.
%\bibitem[GH]{GH} L. G\"ottsche and D.~Huybrechts, \textit{Hodge numbers of moduli spaces of stable bundles on K3 surfaces}, Int.~J.~Math.~7 (1996) 359--372.
\bibitem[GK1]{GK1} L.~G\"ottsche and M.~Kool, \textit{Virtual refinements of the Vafa-Witten formula}, Comm.~Math.~Phys.~376 (2020) 1--49.
\bibitem[GK2]{GK2} L.~G\"ottsche and M.~Kool, \textit{A rank 2 Dijkgraaf-Moore-Verlinde-Verlinde formula}, Comm.~Numb.~Theor.~and Phys.~13 (2019) 165--201.
\bibitem[GK3]{GK3} L.~G\"ottsche and M.~Kool, \textit{Refined $\mathrm{SU}(3)$ Vafa-Witten invariants and modularity}, Pure and Appl.~Math.~Quart.~14 (2018) 467--513.
%\bibitem[GK]{GK} A.~Gholampour and M.~Kool, \textit{Stable reflexive sheaves and localization}, to appear in Jour.~of Pure and Applied~Alg., arXiv:1308.3688.
%\bibitem[GKY1]{GKY1} A.~Gholampour, M.~Kool, and B.~Young, \textit{Rank 2 sheaves on toric 3-folds: classical and virtual counts}, to appear in IMRN, arXiv:1509.03536.
%\bibitem[GKY2]{GKY2} A.~Gholampour, M.~Kool, and B.~Young, work in progress.
\bibitem[GNY1]{GNY1} L.~G\"ottsche, H.~Nakajima, and K.~Yoshioka, \textit{Instanton counting and Donaldson invariants}, J.~Diff.~Geom.~80 (2008) 343--390.
\bibitem[GNY2]{GNY2} L.~G\"ottsche, H.~Nakajima, and K.~Yoshioka, \textit{K-theoretic Donaldson invariants via instanton counting}, Pure Appl.~Math.~Quart.~5 (2009) 1029--1111.
\bibitem[GNY3]{GNY3} L.~G\"ottsche, H.~Nakajima, and K.~Yoshioka, \textit{Donaldson = Seiberg-Witten from Mochizuki's formula and instanton counting}, Publ.~Res.~Inst.~Math.~Sci.~47 (2011) 307--359.
\bibitem[GS]{GS} L.~G\"ottsche and W.~Soergel, \textit{Perverse sheaves and the cohomology of Hilbert schemes of smooth algebraic surfaces}, Math.~Ann.~296 (1993) 235--245.
\bibitem[GY]{GY} L.~G\"ottsche and Y.~Yuan, \textit{Generating functions for $K$-theoretic Donaldson invariants and Le Potier's strange duality}, J.~Alg.~Geom.~28 (2019) 43--98. 
%\bibitem[Got]{Got} L.~G\"ottsche, \textit{A conjectural generating function for numbers of curves on surfaces}, Comm.~Math.~Phys.~196 523-533 (1998).
%\bibitem[Got1]{Got1} L.~G\"ottsche, \textit{Change of polarization and Hodge numbers of moduli spaces of torsion free sheaves on surfaces}, Math.~Z.~223 (1996) 247--260.
\bibitem[GP]{GP} T.~Graber and R.~Pandharipande, \textit{Localization of virtual classes}, Invent.~Math.~135 (1999) 487--518. 
\bibitem[GDP]{GDP} S.~Gukov and Du Pei, \textit{Equivariant Verlinde formula from fivebranes and vortices}, Comm.~Math.~Phys.~355 (2017) 1--50.
%\bibitem[GZ]{GZ} L.~G\"ottsche and D.~Zagier, \textit{Jacobi forms and the structure of Donaldson invariants for 4-manifolds with $b_+=1$}, Selecta Math.~4 (1998) 69--115.
%\bibitem[GV1]{GV1} R.~Gopakumar and C.~Vafa, \textit{M-theory and topological strings---I}, hep-th/9809187.
%\bibitem[GV2]{GV2} R.~Gopakumar and C.~Vafa, \textit{M-theory and topological strings---II}, hep-th/9812127.
%\bibitem[Har]{Har} R.~Hartshorne, \textit{Algebraic geometry}, Springer-Verlag (1977).
%\bibitem[Har1]{Har1} R.~Hartshorne, \textit{Stable vector bundles of rank 2 on $\PP^3$}, Math.~Ann.~\textbf{238} (1978) 229--280.
%\bibitem[Har2]{Har2} R.~Hartshorne, \textit{Stable reflexive sheaves}, Math.~Ann.~\textbf{254} (1980) 121--176.
\bibitem[H-L]{H-L} D.~Halpern-Leistner, \textit{The equivariant Verlinde formula on the moduli of Higgs bundles}, arXiv:1608.01754.
%\bibitem[Huy]{Huy}  D.~Huybrechts,  \textit{Compact hyper-K\"ahler manifolds: basic results},  Invent.~Math.~135 (1999) 63--113.
\bibitem[HL]{HL} D.~Huybrechts, M.~Lehn, \textit{The geometry of moduli spaces of sheaves}, Cambridge University Press (2010).
%\bibitem[Hua]{Hua} L.~K.~Hua, \textit{Introduction to number theory}, Springer (1982).
%\bibitem[HTT]{HTT} Y.~Hinohara, K.~Takahashi, H.~Terakawa, \textit{On tensor products of $k$-very ample line bundles}, Proc.~Amer.~Math.~Soc.~133 687--692 (2004).
%\bibitem[HM]{HM} J.~Harris, I.~Morrison, \textit{Moduli of curves}, Springer-Verlag (1998).
%\bibitem[HT]{HT} D.~Huybrechts and R.~P.~Thomas, \textit{Deformation-obstruction theory for complexes via Atiyah and Kodaira-Spencer classes}, Math.~Ann.~{\bf{346}} (2010) 545--569. arXiv:0805.3527
%\bibitem[Iar]{Iar} A.~Iarrobino, \textit{Punctual Hilbert schemes}, Bull.~Amer.~Math.~Soc.~78 819--823 (1972).
%\bibitem[Ill]{Ill} L.~Illusie, \textit{Complexe cotangent et d\'eformations I}, Lecture Notes in Math.~{\bf{239}} Springer-Verlag (1971).
%\bibitem[Joy1]{Joy1} D.~Joyce, \textit{Configurations in abelian categories. I. Basic properties and moduli stacks}, Adv.~Math.~\textbf{203} (2006) 194--255.
%\bibitem[Joy2]{Joy2} D.~Joyce, \textit{Configurations in abelian categories. II. Ringel-Hall algebras}, Adv.~Math.~\textbf{210} (2007) 635--706.
%\bibitem[Joy3]{Joy3} D.~Joyce, \textit{Configurations in abelian categories. III. Stability conditions and identities}, Adv.~Math.~\textbf{215} (2007) 153--219.
%\bibitem[Joy]{Joy} D.~Joyce, \textit{Configurations in abelian categories. IV. Invariants and changing stability conditions}, Adv.~Math.~217 (2008) 125--204.
%\bibitem[Jun]{Jun} Jun Li, \textit{A note on enumerating rational curves in a K3 surface}, in ``Geometry and nonlinear partial differential equations'' AMS/IP Studies in Adv.~Math.~29 (2002).
%\bibitem[JS]{JS} D.~Joyce and Y.~Song, \textit{A theory of generalized {D}onaldson-{T}homas invariants}, Memoirs of the AMS (2012). arXiv:0810.5645.
%\bibitem[KL1]{KL1} Y.-H.~Kiem and J.~Li, \textit{Gromov-Witten invariants of varieties with holomorphic 2-forms}, arXiv:0707.2986.
%\bibitem[KL2]{KL2} Y.-H.~Kiem and J.~Li, \textit{Localizing virtual cycles by cosections}, JAMS {\bf{26}} (2013) 1025--1050. arXiv:1007.3085
%\bibitem[Kly1]{Kly1} A.~A.~Klyachko, \textit{Equivariant bundles on toral varieties}, Math.~USSR Izvestiya \textbf{35} (1990) 337?375.
%\bibitem[Kim]{Kim}S.-O.~Kim, \textit{Noether-Lefschetz locus for surfaces}, Trans.~Amer.~Math.~Soc.~324 (1991) 369--384. 
\bibitem[JK]{JK} Y.~Jiang and M.~Kool, \textit{Twisted sheaves and $\mathrm{SU}(r) / \Z_r$ Vafa-Witten theory}, arXiv:2006.10368.
\bibitem[Kir]{Kir} F.~Kirwan, \textit{The cohomology rings of moduli spaces of bundles over Riemann surfaces}, J.~Amer.~Math.~Soc.~5 (1992) 853--906. 
%\bibitem[Kly]{Kly} A.~A.~Klyachko, \textit{Vector bundles and torsion free sheaves on the projective plane}, preprint Max Planck Institut f\"ur Mathematik (1991).
%\bibitem[KNS]{KNS} K.~Kodaira, L.~Nirenberg, D.~C.~Spencer, \textit{On the existence of deformations of complex analytic structures}, Ann.~Math.~68 450--459 (1958).
%\bibitem[Koo]{Koo} M.~Kool, \textit{Fixed point loci of moduli spaces of sheaves on toric varieties}, Adv.~Math.~{\bf{227}} (2011) 1700--1755. arXiv:0810.0418
%\bibitem[Koo]{Koo} M.~Kool, \textit{Euler characteristics of moduli spaces of torsion free sheaves on toric surfaces}, Geom.~Ded.~176 (2015) 241--269.
%\bibitem[KS]{KS} K.~Kodaira, D.~C.~Spencer, \textit{A theorem of completeness for complex analytic fibre spaces}, Acta Math.~100 281--294 (1958). 
%\bibitem[KS]{KS} M.~Kontsevich and Y.~Soibelman, \textit{Stability structures, motivic Donaldson-Thomas invariants and cluster transformations}, arXiv:0811.2435.
%\bibitem[KST]{KST} M.~Kool, V.~Shende and R.~P.~Thomas, \textit{A short proof of the G\"ottsche conjecture}, Geom.~Topol.~15 397--406 (2011). arXiv:1010.3211v2.
%\bibitem[Kod]{Kod} K.~Kodaira, \textit{On the structure of compact complex analytic surfaces, I}, Am.~J.~Math.~86 751--798 (1964).
%\bibitem[KT1]{KT1} M.~Kool and R.~P.~Thomas, \textit{Reduced classes and curve counting on surfaces I: theory}, Alg.~Geom.~{\bf{1}} (2014) 334--383. arXiv:1112.3069
%\bibitem[KT2]{KT2} M.~Kool and R.~P.~Thomas, \textit{Reduced classes and curve counting on surfaces II: calculations}, Alg.~Geom.~{\bf{1}} (2014) 384--399. arXiv:1112.3070
%\bibitem[KT4]{KT4} M.~Kool and R.~P.~Thomas, \emph{to appear}.
%\bibitem[Kur]{Kur} M.~Kuranishi, \textit{On the locally complete families of complex analytic structures}, Ann.~Math.~75 536--577 (1962).
\bibitem[Kyn]{Kyn} V.~Kynev, \textit{An example of a simply connected surface of general type for which the local Torelli theorem does not hold}, C.~R.~Acad.~Bulgare Sci.~30 (1977) 323--325.
%\bibitem[Laa]{Laa} T.~Laarakker, \textit{The monopole contribution to Vafa-Witten invariants}, in preparation.
%\bibitem[LL]{LL} J.~M.~F.~Labastida and C.~Lozano, \textit{The Vafa-Witten theory for gauge group $\SU(N)$}, Adv.~Theor.~Math.~Phys.~5 (1999) 1201--1225.
\bibitem[Laa2]{Laa2} T.~Laarakker, \textit{Monopole contributions to refined Vafa-Witten invariants}, to appear in Geom.~Topol., arXiv:1810.00385.
%\bibitem[Lan]{Lan} S.~Lang, \textit{Algebraic number theory}, Springer (1994).
%\bibitem[Lee]{Lee} J.~Lee, \textit{Family Gromov-Witten invariants for K\"ahler surfaces}, Duke Math.~Jour.~ 123 209--233 (2004).
%\bibitem[LLZ]{LLZ} J.~Li, K.~Liu, and J.~Zhou, \textit{Topological string partition functions as equivariant indices}, Asian J.~Math.~10 (2006) 81--114.
%\bibitem[LNS]{LNS} A.~Losev, N.~Nekrasov, and S.~Shatashvili, \textit{Issues in topological gauge theory}, Nuclear Phys.~B 534 (1998) 549--611.
%\bibitem[Lo]{Lo} J.~Lo, \textit{Polynomial Bridgeland stable objects and reflexive sheaves}, Math.~Res.~Lett.~\textbf{19} (2012) 873--885.
%\bibitem[LP]{LP} J.~Lee and T.~Parker, \textit{A structure theorem for the Gromov-Witten invariants of K\"ahler surfaces}, J.~Diff.~Geom.~{\bf{77}} (2007) 483--513. arXiv:0610570  
%\bibitem[LQ1]{LQ1} W.-P.~Li and Z.~Qin, \textit{On blowup formulae for the $S$-duality conjecture of Vafa and Witten}, Invent.~Math.~136 (1999) 451--482.
%\bibitem[LQ2]{LQ2}  W.-P.~Li and Z.~Qin, \textit{On blowup formulae for the $S$-duality conjecture of Vafa and Witten II: the universal functions}, Math.~Res.~Lett.~5 (1998) 439--453.
%\bibitem[LT]{LT} J.~Li and G.~Tian, \textit{Virtual moduli cycles and Gromov-Witten invariants of algebraic varieties}, JAMS {\bf{11}} (1998) 119--174. arXiv:9602007
%\bibitem[Man]{Man} M.~Manetti, \textit{Lectures on deformations of complex manifolds}, Rend.~Mat.~Appl.~24 1--183 (2004).
%\bibitem[Man]{Man} J.~Manschot, \textit{The Betti numbers of the moduli space of stable sheaves of rank 3 on $\mathbb{P}^2$}, Letters in Math.~Phys.~98 (2011) 65--78.  
%\bibitem[MT]{MT} D.~Maulik and R.~P.~Thomas, in preparation.
%\bibitem[Mat]{Mat} A.~Mattuck, \textit{Secant bundles on symmetric products}, Amer.~J.~Math. {\bf{87}} (1965) 779--797.
%\bibitem[MNOP]{MNOP} D.~Maulik, N.~Nekrasov, A.~Okounkov, and R.~Pandharipande, \textit{Gromov-{W}itten theory and {D}onaldson-{T}homas theory, {I}}, Compos.~Math.~142 (2006) 1263--1285. 
%\bibitem[MNOP2]{MNOP2} D.~Maulik, N.~Nekrasov, A.~Okounkov, and R.~Pandharipande, \textit{Gromov-{W}itten theory and {D}onaldson-{T}homas theory, {II}}, Compos.~Math.~{\bf{142}} 1286--1304 (2006). arXiv:math/0406092
\bibitem[Moc]{Moc} T.~Mochizuki, \textit{Donaldson type invariants for algebraic surfaces}, Lecture Notes in Math.~1972, Springer-Verlag, Berlin (2009). 
\bibitem[Mor]{Mor} J.~W.~Morgan, \textit{The Seiberg-Witten equations and applications to the topology of smooth four-manifolds}, Math.~Notes 44, Princeton Univ.~Press (1996).
%\bibitem[MOOP]{MOOP} D.~Maulik, A.~Oblomkov, A.~Okounkov, and R.~Pandharipande, \textit{Gromov-{W}itten/{D}onaldson-{T}homas correspondence for toric 3-folds}, Invent.~Math.~186 435--479 (2011).
%\bibitem[Moo]{Moo} J.D.~Moore, \textit{Lectures on Seiberg-Witten invariants}, Lecture Notes in Mathematics 1629, Springer-Verlag (1996).
%\bibitem[Moz]{Moz} S.~Mozgovoy, \textit{Invariants of moduli spaces of stable sheaves on ruled surfaces}, arXiv:1302.4134.
%\bibitem[MP]{MP}  D.~Maulik and R.~Pandharipande, \textit{New calculations in Gromov-Witten theory}, Pure Appl.~Math.~Q.~{\bf{4}}, Special Issue: In honor of Fedor Bogomolov, (2008) 469--500. arXiv:0601395
%\bibitem[MP2]{MP2} D.~Maulik, R.~Pandharipande, \textit{Gromov-Witten theory and Noether-Lefschetz theory}, arXiv:0705.1653 (2010).
%\bibitem[MPT]{MPT} D.~Maulik, R.~Pandharipande and R.~P.~Thomas, \textit{Curves on K3 surfaces and modular forms}, J.~Topol.~3 937--996 (2010). 
%\bibitem[Mum]{Mum} D.~Mumford, \textit{Lectures on curves on an algebraic surface}, Princeton University Press 1966.
%\bibitem[Nak1]{Nak1} H.~Nakajima, \textit{Instantons on ALE spaces, quiver varieties, and Kac-Moody algebras}, Duke Math.~J.~76 (1994) 365--416.
%\bibitem[Nak2]{Nak2} H.~Nakajima, \textit{Gauge theory on resolutions of simple singularities and simple Lie algebras}, IMRN (1994) 61--74.
%\bibitem[Nek1]{Nek1} N.~Nekrasov,  \emph{Four dimensional holomorphic theories}, PhD Thesis, Princeton (1996).
%\bibitem[Nek2]{Nek2} N.~Nekrasov, \textit{Seiberg-Witten prepotential from instanton counting}, Adv.~Theor.~Math.~Phys. 7 (2003) 831--864. 
%\bibitem[NO]{NO} N.~Nekrasov and A.~Okounkov, \textit{Seiberg-Witten prepotential and random partitions}, The unity of mathematics, 525--596, Progr.~Math.~244, Birkh\"auser, Boston (2006).
\bibitem[NO]{NO} N.~Nekrasov and A.~Okounkov, \textit{Membranes and sheaves}, Alg.~Geom.~3 (2016) 320--369.
%\bibitem[NY1]{NY1} H.~Nakajima and K.~Yoshioka, \textit{Lectures on instanton counting}, Algebraic structures and moduli spaces, CRM Proc.~Lecture Notes vol.~38, AMS Providence RI, 31--101 (2004).
%\bibitem[NY2]{NY2} H.~Nakajima and K.~Yoshioka, \textit{Instanton counting on blowup. I. 4-dimensional pure gauge theory}, Invent.~Math.~162 (2005) 313--355.
%\bibitem[Oko]{Oko} A.~Okounkov, \textit{Lectures on $K$-theoretic computations in enumerative geometry}, arXiv:1512.07363.
%\bibitem[Pan]{Pan} R.~Pandharipande, \textit{Descendents for stable pairs on 3-folds}, arXiv:1703.01747.
%\bibitem[Per]{Per} U.~Persson, \textit{Chern invariants of surfaces of general type}, Compos.~Math.~43 (1981) 3--58.
%\bibitem[Pot]{Pot} J.~Le Potier, \textit{Faisceaux semi-stables de dimension 1 sur le plan projectif}, Rev.~Roumaine Math.~Pures Appl.~38 635--678 (1993).
%\bibitem[Pid]{Pid} V.~Ya.~Pidstrigach, \textit{Deformations of instanton surfaces}, Izv.~Akad.~Nauk SSSR Ser.~Mat.~55 318--338 (1991).
%\bibitem[PP1]{PP1} R.~Pandharipande and A.~Pixton, \textit{Gromov-Witten/Pairs descendent correspondence for toric 3-folds}, Geom.~and Topol.~{\bf{18}} (2014) 2747--2821. arXiv:1203.0468
%\bibitem[PP2]{PP2} R.~Pandharipande and A.~Pixton, \textit{Gromov-Witten/Pairs correspondence for the quintic 3-fold}, arXiv:1206.5490.
%\bibitem[PT1]{PT1} R.~Pandharipande and R.~P.~Thomas, \textit{Curve counting via stable pairs in the derived category}, Invent.~Math.~{\bf{178}} (2009) 407--447. arXiv:0707.2348
%\bibitem[PT2]{PT2} R.~Pandharipande and R.~P.~Thomas, \textit{The 3-fold vertex via stable pairs}, Geom.~and Topol.~{\bf{13}} (2009) 1835--1876. arXiv:0709.3823
%\bibitem[PT3]{PT3} R.~Pandharipande and R.~P.~Thomas, \textit{Stable pairs and BPS invariants}, JAMS {\bf{23}} (2010) 267--297. arXiv:0711.3899
%\bibitem[PT4]{PT4} R.~Pandharipande and R.~P.~Thomas, \emph{in preparation}.
\bibitem[Ram]{Ram} T.~R.~Ramadas, \textit{Factorisation of generalised theta functions II: The Verlinde formula}, Topol.~35 (1996) 641--654. 
%\bibitem[Ran]{Ran} Z.~Ran, \textit{Semiregularity, obstructions and deformations of Hodge classes}, Ann.~Scuola Norm.~Sup.~Pisa Cl.~Sci.~4 28 809--820 (1999).
%\bibitem[Ser]{Ser} J.-P.~Serre, \textit{A course in arithmetic}, Springer (1996).
%\bibitem[She]{She} J.~Shen, \textit{Cobordism invariants of the moduli space of stable pairs}, J.~London Math.~Soc.~94 (2016) 427--446.
%\bibitem[Sie]{Sie} B.~Siebert, \textit{Virtual fundamental classes, global normal cones and Fulton's canonical classes}, in: Frobenius manifolds, ed.~K.~Hertling and M.~Marcolli, Aspects Math.~36 341--358, Vieweg (2004).
%\bibitem[Spi]{Spi} H.~Spielberg, \textit{Une formule pour les invariants de Gromov-Witten des vari\'et\'es toriques}, PhD Thesis Universit\'e Louis Pasteur (1999).
\bibitem[Sze]{Sze} A.~Szenes, \textit{Hilbert polynomials of moduli spaces of rank 2 vector bundles I}, Topol.~32 (1993) 587--597.
%\bibitem[ST]{ST} J.~Stoppa and R.~P.~Thomas, \textit{Hilbert schemes and stable pairs: GIT and derived category wall crossing}, Bull.~Soc.~Math.~France \textbf{139} (2011) 297--339.
%\bibitem[Sta]{Sta} R.~P.~Stanley, \textit{Enumerative combinatorics, Volume 2}, Cambridge Studies in Adv.~Math.~\textbf{62} (2001).
%\bibitem[Gun]{Gun} S.~Gunningham, \textit{Spin Hurwitz numbers and topological quantum field theory}, to appear in Geom.~and Topol. arXiv:1201.1273
%\bibitem[Tau1]{Tau1} C.~H.~Taubes, \textit{The Seiberg-Witten and Gromov invariants}, Math.~Res.~Lett.~2 221--238 (1995).
%\bibitem[Tau2]{Tau2} C.~H.~Taubes, \textit{Gr=SW: counting curves and connections}, J.~Diff.~Geom.~52 453-609 (1999). 
\bibitem[TT1]{TT1} Y.~Tanaka and R.~P.~Thomas, \textit{Vafa-Witten invariants for projective surfaces I: stable case}, Jour.~Alg.~Geom.~(2019) doi.org/10.1090/jag/738.
\bibitem[Tha]{Tha} M.~Thaddeus, \textit{Stable pairs, linear systems and the Verlinde formula}, Invent.~Math.~117 (1994) 317--353.
\bibitem[Tho]{Tho} R.~P.~Thomas, \emph{Equivariant K-theory and refined Vafa-Witten invariants}, to appear in Comm.~Math.~Phys., arXiv:1810.00078.
%\bibitem[TT2]{TT2} Y.~Tanaka and R.~P.~Thomas, \textit{Vafa-Witten invariants for projective surfaces II: semistable case}, arXiv:1702.08488.
%\bibitem[Tod]{Tod} Y.~Toda, \textit{Curve counting theories via stable objects I.~DT/PT correspondence}, J.~Amer.~Math.~Soc.~23 1119--1157 (2010). 
%\bibitem[Tod]{Tod} Y.~Toda, \textit{Hall algebras in the derived category and higher rank DT invariants}, arXiv:1601.07519. 
%\bibitem[TW]{TW} C.~Teleman and C.~T.~Woodward, \textit{The index formula on the moduli of $G$-bundles}, Annals of Math.~(2009) 495--527.
\bibitem[VW]{VW} C.~Vafa and E.~Witten, \textit{A strong coupling test of $S$-duality}, Nucl.~Phys.~B 431 (1994) 3--77.
\bibitem[Ver]{Ver} E.~Verlinde, \textit{Fusion rules and modular transformations in $2d$ conformal field theory}, Nucl.~Phys.~B 300 (1988) 360--376.
%\bibitem[Voi1]{Voi1} C.~Voisin, \textit{Hodge theory and complex algebraic geometry I}, Cambridge University Press (2002).
%\bibitem[Voi]{Voi} C.~Voisin, \textit{Hodge loci}, Handbook of moduli (to appear).
%\bibitem[Wei]{Wei} T.~Weist, \textit{Torus fixed points of moduli spaces of stable bundles of rank three}, J.~Pure Appl.~Algebra 215 (2011) 2406--2422.
%\bibitem[Wit]{Wit} E.~Witten, \textit{Monopoles and four-manifolds}, Math.~Res.~Lett.~1 769--796 (1994). \\
%\bibitem[Wit]{Wit} E.~Witten, \textit{AdS/CFT correspondence and topological field theory}, JHEP 9812 (1998) 012.
%\bibitem[Yos1]{Yos1} K.~Yoshioka, \textit{The Betti numbers of the moduli space of stable sheaves of rank 2 on $\PP^2$}, J.~Reine Angew.~Math.~453 (1994) 193--220.
%\bibitem[Yos2]{Yos2} K.~Yoshioka, \textit{The Betti numbers of the moduli space of stable sheaves of rank 2 on a ruled surface}, Math.~Ann.~302 (1995) 519--540.
%\bibitem[Yos3]{Yos3} K.~Yoshioka, \textit{Number of $\FF_q$-rational points of the moduli of stable sheaves on elliptic surfaces}, Moduli of vector bundles, editor: M.~Maruyama, Lect.~Notes in Pure and Appl.~Math.~179, Marcel Dekker, New York (1996).
%\bibitem[Yos4]{Yos4} K.~Yoshioka, \textit{Betti numbers of moduli of stable sheaves on some surfaces}, Nucl.~Phys.~B (Proc.~Suppl.)
%{\bf 46}, (1996) 263--268.
%\bibitem[Yos]{Yos} K.~Yoshioka, \textit{Some examples of Mukai's reflections on K3 surfaces},  J.~Reine~Angew.~Math.~515 (1999) 97--123. 
\bibitem[Zag1]{Zag1} D.~Zagier, \textit{Elementary aspects of the Verlinde formula and of the Harder-Narasimhan-Atiyah-Bott formula}, Proc.~of the Hirzebruch 65 Conf.~on Alg.~Geom., Israel Math.~Conf.~Proc.~9., Ramat Gan (1996) 455--462.
\bibitem[Zag2]{Zag2} D.~Zagier, \textit{On the cohomology of moduli spaces of rank two vector bundles over curves}, in: The Moduli Space of Curves 533--563, Progress in Math.~129,
 Birkh\"auser (1995).
\end{thebibliography}
\end{document}